\documentclass[a4paper,11pt]{amsart}
\usepackage{amsmath,inputenc,euscript,amssymb,geometry}
\usepackage{graphicx}
\usepackage{amssymb}
\usepackage{latexsym}
\usepackage{amssymb,amsbsy,amsmath,amsfonts,amssymb,amscd}
\usepackage{hyperref}

\newcommand{\D}{\displaystyle}
\newtheorem{lemma}{Lemma}[section]
\newtheorem{theorem}[lemma]{Theorem}
\newtheorem{prop}[lemma]{Proposition}

\newtheorem{remark}[lemma]{Remark}

\begin{document}
\title[]{Stability of the selfsimilar dynamics\\ of a vortex filament}
\author[V. Banica]{Valeria Banica}
\address[V. Banica]{Laboratoire Analyse et probabilit\'es (EA 2172)\\D\'eptartement de Math\'ematiques\\ Universit\'e d'Evry, 23 Bd. de France, 91037 Evry\\ France, Valeria.Banica@univ-evry.fr} 

\author[L. Vega]{Luis Vega}
\address[L. Vega]{Departamento de Matematicas, Universidad del Pais Vasco, Aptdo. 644, 48080 Bilbao, Spain, luis.vega@ehu.es}

\maketitle

\begin{abstract}
In this paper we continue our investigation about selfsimilar solutions of the vortex filament equation, also known  as the binormal flow (BF) or the localized induction equation (LIE). Our main result is the stability of the selfsimilar dynamics  of small pertubations of  a given selfsimilar solution. The proof relies on finding precise asymptotics in space and time for the tangent and the normal vectors of the perturbations. A main ingredient in the proof  is the control of the evolution of weighted norms for a cubic 1-D Schr\"odinger equation, connected to the binormal flow by Hasimoto's transform.

\end{abstract}

\tableofcontents

\section{Introduction}
We consider the geometric PDE 
\begin{equation}\label{binormal}
\chi_t=\chi_x\land\chi_{xx}
\end{equation}
that is usually known as  the binormal flow (BF) or the localized induction equation (LIE).
Above $\chi=\chi(t,x)\in\mathbb R^3$, $x$ denotes the arclength parameter and $t$ the time variable. Using the Frenet frame, the above equation can be written as
\begin{equation*}\label{binormal1}
\chi_t=c\,b,
\end{equation*}
where $c$ is the curvature of the curve and $b$ its binormal vector. This geometric flow was proposed by Da Rios in 1906 \cite{DaR} as an approximation of the evolution 
of a vortex filament in a 3-D incompressible inviscid fluid (see also \cite{ArHa}). 
We refer the reader to  \cite{AlKuOk}, \cite{Ba}, \cite{Sa} and \cite{MaBe} for an analysis and discussion about the limitations of this model and to \cite{Ri} for a survey about Da Rios' work. Local well-posedness results for the binormal flow were obtained when curvature and torsion are in high order Sobolev spaces, 
see \cite{Ha, Ko,FuMi}. For less regular closed curves Jerrard and Smets obtained recently in  \cite{JeSm1,JeSm2} a result of global existence for a weak version of the binormal flow. They also proved a weak-strong uniqueness property, as long as self-intersections do not occur.

The selfsimilar solutions with respect to scaling of \eqref{binormal} are easily found by first fixing the ansatz 
\begin{equation}\label{ansatz}
\chi(t,x)=\sqrt{t}\,G\left(\frac {x}{\sqrt{t}}\right).
\end{equation}
Plugging this ansatz in \eqref{binormal} and eliminating time one obtains the ODE
\begin{equation}\label{eq1.1}
\frac12 G-\frac s2 G^\prime=G^\prime\land G^{\prime\prime}.
\end{equation}
After differentiation in $s$, calling $T(s)=G^\prime(s)$, and using the system of Frenet equations we get
$$-\frac s2\,c\,n=-\frac s2 \,T^\prime=T\land T^{\prime\prime}= c_s\, b-c\tau\, n,$$
where $n$ denotes the normal vector and $\tau$ the torsion. Hence we conclude that the selfsimilar solutions are characterized by the geometric conditions
\begin{equation*}\label{ctauselfsim}
c(s)=a,\quad\qquad\tau(s)=\frac{s}{2},
\end{equation*}
for a parameter $a\in\mathbb R$ (see \cite{Bu}). The case $a=0$ gives a straight line so that we can assume without loss of generality that $a>0$. Given $a$, the corresponding solutions of \eqref{binormal} are  unique modulo a translation and a rotation. Indeed, assume that the Frenet frame $(T,n,b)$ at $s=0$ is the identity matrix, so that from \eqref{eq1.1} we obtain $G(0)=2a\,b(0)=(0,0,2a)$. Call $G_a$ the corresponding curve and $T_a$ its unit tangent. Hence
we conclude that
$$\chi_a(t,x)= \sqrt{t}\,G_a\left(\frac {x}{\sqrt{t}}\right)$$
is a solution of \eqref{binormal} for $t>0$ and that
$$\mathrm{T}_a(t,x)= T_a\left(\frac {x}{\sqrt{t}}\right)$$
solves  for $t>0$
$$\mathrm{T}_t=\mathrm{T}\land\mathrm{T}_{xx},\qquad |\mathrm{T}|=1,$$
usually known as the Schr\"odinger map into the $\mathbb S^2$ sphere. We denote by 
$$N(t,x)=(n+ib)(t,x)\,e^{i\int_0^x \tau(t,s)\,ds},$$
the ``parallel" normal vector. The properties of this frame  will be described in \S\ref{paralsection}.

 It was proved in \cite{GRV} that there exist $A^\pm_a\in\mathbb S^2$ and $B_a^\pm\in\mathbb C^2$ such that for $x>0$ (and similarly for $x<0$), 
\begin{equation}
\label{0.1}\left|\chi_a(t,x)-A_a^+\left(x+2a\frac tx\right)-4a\frac{t}{x^2}\,n_a(t,x)\right|\leq C\left(\frac{\sqrt {t}}{x}\right)^3,
\end{equation}
\begin{equation}
\label{0.2}|T_a(t,x)-A_a^+|\leq C\frac{\sqrt {t}}{x},
\end{equation}
\begin{equation}
\label{0.3}
\left|N_a(t,x)-B_a^+\,e^{ia^2\log\frac{\sqrt t}{x}}\right|\leq C\frac{\sqrt {t}}{x}.
\end{equation}
Moreover, $A_a^\pm\perp B_a^\pm$ and if we define $\theta$ as the angle between $ A_a^+$ and $-A_a^-$
\begin{equation}
\label{0.4}\sin\frac{\theta}{2}=e^{-\pi\frac{a^2}{2}}.
\end{equation}
Also the coordinates of $A_a^\pm$ and $B_a^\pm$ are given explicitly in terms of Gamma functions involving the parameter $a$ (see formula (55), (57), (47), (48), and (69) in \cite{GRV}). In particular we can define at time zero for $x>0$
\begin{equation}
\label{0.5}\lim_{t\rightarrow 0} T_a(t,x)=A_a^{+}\quad,\quad  \lim_{t\rightarrow 0} N_a(t,x)\,e^{-ia^2\log\frac{\sqrt t}{x}}=B^{+}_a,
\end{equation}
and similarly when $x<0$ using in that case $(A_a^{-},B^{-}_a).$

The reader can find in \cite{GRV}  some pictures of $G_a$ and $\chi_a$ for different values of $a$.
Also in \cite{dHGV} some numerical simulations are considered. In the figure 1.1 of that paper it is showed the remarkable similarity, at least at the qualitative level, of $\chi_a$ and the vortex filaments that appear in the flow of a fluid traversing a delta wind -see \cite{ON}. We also encourage the reader to look at the selfsimilar shape of the smoke rings in the picture 107 in \cite{Dy}. It seems from these pictures and from the numerical simulations, that the selfsimilar dynamics of these vortex filaments are rather stable. 

In our two previous papers \cite{BV1} and \cite{BV2} we obtain some results about the stability and the instability of the solutions $\chi_a$. Our approach is based on the so-called Hasimoto transformation. In \cite{Ha} the "filament function"
\begin{equation}
\label{eq1.2}
\psi(x,t)=c(x,t)\,e^{i\int_0^x \tau(s,t)\,ds}
\end{equation}
is defined and it is proved that if $c$ and $\tau$ are the curvature and the torsion respectively of a solution $\chi(x,t)$ of \eqref{binormal}, then $\psi$ solves the focusing cubic non-linear Schr\"odinger equation (NLS)
$$i\psi_t+\psi_{xx}+\frac \psi 2\left(|\psi|^2-A(t)\right)=0$$ 
for some real function $A(t)$ that depends on $c(0,t)$ and $\tau(0,t)$. In the particular case of $\chi_a$ we have that for $t>0$
\begin{equation}
\label{eq1.3}
\psi_a(x,t)=a\frac {e^{i\frac {x^2}{4t}}}{\sqrt {t}} 
\end{equation}
and $A(t)=\frac {|a|^2}{t}$. 
Hasimoto's transformation can be performed  only in case of nonvanshing curvature. This obstruction has been avoided by Koiso \cite{Ko} by using another frame than the Frenet one. 

Notice that 
$$\int |\psi_a(x,t)|^2\,dx=+\infty,$$
so that $L^2(\mathbb R)$ is not the right functional setting to study $\psi_a$. It is natural to consider the so-called pseudoconformal transformation of $\psi$ defining a new unknown $\,v\,$ as
\begin{equation}\label{calT}
\psi(t,x)=\mathcal {T}v(t,x)=\frac{e^{i\frac{x^2}{4t}}}{\sqrt{t}}\overline{v}\left(\frac 1t,\frac xt\right).
\end{equation}
Then $\,v\,$ solves 
\begin{equation}\label{GP1}
iv_t+v_{xx}+\frac {1}{2t}\left(|v|^2-a^2\right)v=0,\\
\end{equation}
and $v_a=a$ is the particular solution that corresponds to $\psi_a$. A natural quantity associated to \eqref{GP1} is the normalized energy (see \cite{BV0})
\begin{equation*}\label{energy}E(v)(t)=\frac{1}{2}\int |v_x(t)|^2\,dx-\frac{1}{4t}\int(|v(t)|^2-a^2)^2\,dx.
\end{equation*}
An immediate calculation gives that
\begin{equation*}\label{GP2}\partial_{t}E(v)(t)-\frac{1}{4t^2}\int(|v|^2-a^2)^2\,dx=0,
\end{equation*}
and in particular $E(v_a)=0.$\\

The binormal flow \eqref{binormal} is an equation that is reversible in time. If we want to study perturbations of $\chi_a$ one possibility is to go forward in time starting at time $t=0$ with  a datum close to
\begin{equation}
\label{1.4}
\chi_a(0,x)=\renewcommand{\arraystretch}{1.4}\begin{cases}
  A_a^+ x&x\ge 0\\
   A_a^- x&x\le 0,\\
  \end{cases}\renewcommand{\arraystretch}{1}
 \end{equation}
 and to construct a solution up to say time $t=1$.
Another possibility is to  give a datum at time $t=1$ close to $G_a$  and go backwards in time up to time $t=0$.

In terms of $v$ these two possibilities are rephrased as follows. First we write
$$v=a+u,$$
so that $\,u\,$ has to be a solution of
\begin{equation}\label{NLS}
iu_t+u_{xx}+\frac{a+u}{2t}(|a+u|^2-a^2)=0.
\end{equation}
In particular
$$u(1,x)=e^{i\frac{x^2}{4}}\overline{\psi}(1,x)-a.$$
Then notice that the pseudoconformal transformation sends the interval of time $[0,1]$ into the interval $[1,\infty)$. So that the first possibility, that is to say to go forward in time in \eqref{binormal}, amounts to give a small asymptotic state  at time infinity and construct a solution for $t\ge1$ of \eqref{NLS} that remains close to it in an appropriate sense. The second possibility is to solve the initial value problem of \eqref{NLS}  with some small datum at $t=1$ and to prove the existence of a scattering state at infinity with a size controlled by that one of the initial datum. In \cite{BV1} and \cite{BV2} we study the two problems. Finally let us notice that long time asymptotics were studied for equations    
with a common point with (12) in terms of the nonlinearity, like the linear     
Schrodinger equation with a time depending potential (see ch.4 of               
\cite{DeGe}), the 1-D cubic NLS (\cite{Oz},\cite{Ca},\cite{HaNa}),         
the 2-D Gross-Pitaevskii equation (\cite{GuNaTs}), and the 2-D       
quadratic NLS (\cite{MoToTs},\cite{ShTo},\cite{GeMaSh}). However the framework, approach and results        
for (12) are quite different.

More concretely in \cite{BV2} we consider small initial data at time $t=1$, $u_1(x)=u(1,x) \in X_1^\gamma$, $0<\gamma<\frac 14$, where
\begin{equation}\label{Xtau}
\|f\|_{X_{t_0}^\gamma}=\frac{1}{t_0^\frac 14}\|f\|_{L^2}+\frac{t_0^\gamma}{\sqrt{t_0}}\||\xi|^{2\gamma}\hat{f}(\xi)\|_{L^\infty(\xi^2\leq 1)},
\end{equation}
and $\hat f$ denotes the Fourier Transform of $f$. 
The smallness of $u_1$ in $X^\gamma_1$ is with respect to $a$, with a nonlinear dependence. 
In Theorem 1.1 of \cite{BV2} we prove that there exists $f_+\in L^2$ for which
\begin{equation}\label{scatt}
\left\|u(t)-e^{i\frac {a^2}2\log t}e^{i(t-1)\partial_x^2}f_+\right\|_{L^2}\leq \frac{C(a,u_1)}{t^{\frac14-(\gamma+\delta)}}\,\|u_1\|_{X_1^\gamma}\underset{t\longrightarrow\infty}{\longrightarrow} 0,
\end{equation}
for any $0<\delta<1/4-\gamma$. Finally, the asymptotic state  $f_+$ satisfies for all $\xi^2\leq 1$ the estimate
$$|\xi|^{2(\gamma+\delta)}|\hat{f_+}(\xi)|\leq C(a,\delta)\,\|u_1\|_{X_1^{\gamma}}.$$
Then, from a solution $u(t)$ of \eqref{NLS} one constructs a solution $\chi(t)$ of the binormal flow \eqref{binormal} by setting $\psi=\mathcal T(a+u)$ and solving the Frenet system with curvature $c(t,x)=|\psi(t,x)|$ and torsion $\tau(t,x)=\partial_x\arg\psi(t,x)$ (see for instance \cite{BV1}).
Notice that since we are considering small perturbations of the selfsimilar solutions, the curvature function $c(t,x)$ does not vanish and Hasimoto's transform make sense.

The main purpose of this paper is to prove that most of the properties \eqref{0.1}-\eqref{0.5} that describe the dynamics of the selfsimilar  solution $\chi_a$ still hold for the perturbations $\chi$ under some extra conditions on $u_1$. As a consequence the selfsimilar dynamics going backwards in time remain stable under small perturbations of $G_a(x)=\chi_a(1,x)$.
\medskip

Our main result is the following one.

\begin{theorem}\label{errortang}
Let $a>0$ and let $u_1$ be a function such that $\partial_k u_1$ is small with respect to $a$ in $X_1^\gamma$, with $0\leq\gamma\leq\frac 14$, for $0\leq k\leq 4$. Moreover, suppose that $xu_1, x\partial_x u_1$ are in $L^2$, without smallness condition. Given $\chi_1(0)\in\mathbb R^3$ and $\partial_s\chi_1( 0)\in\mathbb S^2$, let $\chi_1(x)$ be the corresponding curve with filament function $ae^{i\frac{x^2}{4}}+u_1(x)e^{i\frac{x^2}{4}}$. Then the unique Lipschitz solution $\chi(t,x)$ of the binormal flow for $0\leq t\leq 1$ with $\chi(1,x)=\chi_1(x)$ constructed in \cite{BV2} enjoys the following properties. The choice between $\pm$ will be determined by $|x|=\pm x$. 

\noindent
(i) {\bf{\em{Asymptotics in space for the tangent vector and the normal vectors at fixed time:}}}
There exist $T^{\pm\infty}\in\mathbb S^2$ and $N^{\pm\infty}\in\mathbb C^3$ such that for all $0<t\leq 1$, and $x\neq 0$\footnote{In the following the relations for $x>0$ will involve $T^{+\infty}$ and $N^{+\infty}$ and the ones for $x<0$ will involve $|x|$, $T^{-\infty}$ and $N^{-\infty}$.},
$$|T(t,x)-T^{\pm\infty}|\leq C\|\partial_xu_1\|_{X^\gamma_1}\,\frac 1{\sqrt{x}}+C(a+\|u_1\|_{X^\gamma_1}+\|\partial_xu_1\|_{X^\gamma_1})\frac{\sqrt t}{x},$$
$$\left|N(t,x)-N^{\pm\infty}\,e^{ia^2\log\frac{\sqrt t}{x}}\right|\leq C\|\partial_xu_1\|_{X^\gamma_1}\,\frac 1{\sqrt{x}}+C(1+a^2)(a+\|u_1\|_{X^\gamma_1}+\|\partial_xu_1\|_{X^\gamma_1})\frac{\sqrt t}{x}$$
$$+C(a+\|u_1\|_{X^\gamma_1}+\|\partial_xu_1\|_{X^\gamma_1})\|u_1\|_{X^\gamma_1}\,\frac{\sqrt t}{\sqrt x}+C\,(a^2+a^4)\,\frac{t}{x^2}.$$
(ii) {\bf{\em{Further informations on the tangent vector:}}} For all $x\neq 0$ and all $t>0$ ,
$$|T(t,x)-T^{\pm\infty}| \leq C(\|u_1\|_{X^\gamma_1}+\|\partial_xu_1\|_{X^\gamma_1})+C(a+\|u_1\|_{X^\gamma_1}+\|\partial_xu_1\|_{X^\gamma_1})\frac{\sqrt t}{x}$$
$$+C(a)(\|u_1\|_{X^\gamma_1}+\|\partial_xu_1\|_{X^\gamma_1}+\|xu_1\|_{L^2})\,\frac{t^\frac 14}{x}.$$
(iii) {\bf{\em{Formation of a corner at time 0:}}} For all $x\neq 0$ 
$$|\chi(0,x)-\chi(0,0)-T^{\pm\infty} x|\leq C\sqrt t+|x| \left(C(\|u_1\|_{X^\gamma_1}+\|\partial_xu_1\|_{X^\gamma_1})\right.$$
$$\left. +C(a+\|u_1\|_{X^\gamma_1}+\|\partial_xu_1\|_{X^\gamma_1})\frac{\sqrt t}{x}+C(a)(\|u_1\|_{X^\gamma_1}+\|\partial_xu_1\|_{X^\gamma_1}+\|xu_1\|_{L^2})\,\frac{t^\frac 14}{x}\right).$$
(iv) {\bf{\em{Existence of a limit for the tangent at time 0:}}} For all $x\neq 0$ there is a limit for $T(t,x)$ as $t$ goes to zero and 
$$|T(t,x)-T(0,x)|=O(t^{\frac 16^-}).$$
Moreover,
 $$T_x(0)\in L^1\cap L^2(\mathbb R\setminus \{0\}).$$
(v) {\bf{\em{The exact value of the angle of the corner:}}} The angle of the self-similar solutions is recovered at time $0$,
$$\sin\frac{(T(0,0^+),-T(0,0^-))}{2}=e^{-\pi\frac{a^2}{2}}.$$
More precisely, modulo a rotation, we recover at the singularity point the self-similar structure
$$\lim_{x\rightarrow 0^\pm} \lim_{t\rightarrow 0} T(t,x)=A_a^\pm\quad,\quad  \lim_{x\rightarrow 0^\pm} \lim_{t\rightarrow 0} N(t,x)\,e^{-ia^2\log\frac{\sqrt t}{x}}=B_a^\pm.$$
\end{theorem}
\medskip

\begin{remark}\label{remcond0} The above theorem gives a precise result about the dynamics of the perturbed filament in the selfsimilar region $|x|>\sqrt t$ for $1\geq t\geq0$. In particular it proves the existence of a natural binormal frame associated to the curve $\chi(0,x)$ even though it has a corner at $x=0$. For doing this it is crucial to be able to use that $u(t)$ belongs to weighted $L^2$ spaces. All the analysis follows from the property that the tangent vectors
of the perturbed filament are fixed for $x=\pm \infty$ and $1\geq t>0$. Once this is proved we integrate the Frenet frame, in fact we use the so-called parallel frame that turns out to be much more convenient,  from $\pm \infty$ to $|x|>\sqrt t$. This is enough for our purposes.
\end{remark}

\begin{remark}\label{remcond}
 We do not obtain anything new in the interior region $|x|<\sqrt t$. At this respect we recall Theorem 1.4 of \cite{BV1}. In that theorem it is proved that if the zero Fourier mode of the asymptotic state that determines  $u(t,x)$ vanishes in an appropriate sense, then $\chi(t,x)$ remains close to $\chi_a(x,t)$ together with their respective Frenet frames also in the region $|x|\leq \sqrt t$. In particular the trajectory $\chi(t,0)$ and the one of the frame $(T, n, b)(t,0)$ remain close to
$\chi_a(t,0)$ and to the identity matrix.  As a consequence, a very natural question is to characterize the asymptotic states of solutions $u(t)$ that belong to weighted $L^2$ spaces. It turns out that the answer is more delicate than what one could expect so that we will study it in a forthcoming paper. Finally, recall that in the appendix B2 of \cite{BV2} it is proved that the zero Fourier modes of solutions $u(t)$ that are in weighted $L^2$ spaces typically grow logarithmically in time.
\end{remark}

The paper is organized as follows. In section 2 we introduce the parallel frame and its connection with the Frenet frame. The proof of our theorem is given in sections 3-5. In the appendix we show some estimates about the evolution in time of the norms of weighted $L^2$ spaces for the solutions $u(t)$ of \eqref{NLS}, needed in Lemma \ref{L1est}.\\

{\bf{Acknowledgements:}} The authors are grateful to the referee's suggestions of improvements of the presentation of the paper. 
First author was partially supported by the French ANR project  R.A.S. ANR-08-JCJC-0124-01. The second author was partially supported by the grants UFI 11/52, MTM 2011-24054 of MEC (Spain) and FEDER. \smallskip

\section{The parallel frame}\label{paralsection}
In the original work of Hasimoto \cite{Ha}, for performing the transformation \eqref{eq1.2} a non-vanishing condition on the curvature was imposed. This condition has been removed by Koiso \cite{Ko} who worked with another frame than the Frenet one. Although in our case the curvature does not vanish for small perturbations of the selfsimilar solutions, we shall take advantage of this Hasimoto-type link built between the cubic 1-D NLS and the binormal flow \eqref{binormal}. We shall detail it below. The use of this frame makes the calculations of the next sections much shorter.

Given $a>0$ we start with a solution of 
\begin{equation}\label{NLSa}
i\psi_t+\psi_{xx}+\frac \psi 2\left(|\psi|^2-\frac{a^2}{t}\right)=0.
\end{equation}
As explained in the Introduction we shall consider
$$\psi(t,x)=\frac{e^{i\frac{x^2}{4t}}}{\sqrt t}(a+\overline u)\left(\frac 1t,\frac xt\right).$$
We define
$$\alpha(t,x)=\Re\psi(t,x)\,\,\,,\,\,\,\beta(t,x)=\Im\psi(t,x).$$
Then, for a given orthonormal frame $(T,e_1,e_2)(0,0)$ as initial data we define an orthonormal frame $(T,e_1,e_2)(t,x)$ by imposing
$$
\left(\begin{array}{c}
T\\e_1\\e_2
\end{array}\right)_x(t,x)=
\left(\begin{array}{ccc}
0 & \alpha & \beta \\ -\alpha & 0 & 0 \\ -\beta &  0 & 0 
\end{array}\right)
\left(\begin{array}{c}
T\\e_1\\e_2
\end{array}\right)(t,x)\,\,\,,$$
and $$\left(\begin{array}{c}
T\\e_1\\e_2
\end{array}\right)_t(t,0)=
\left(\begin{array}{ccc}
0 & -\beta_x & \alpha_x \\ \beta_x & 0 & -\frac{|\psi|^2}{2}+\frac{a^2}{2t} \\ -\alpha_x &  \frac{|\psi|^2}{2}-\frac{a^2}{2t} & 0 
\end{array}\right)
\left(\begin{array}{c}
T\\e_1\\e_2
\end{array}\right)(t,0).$$
We want to compute $T_t(t,x)$.
For all $(t,x)$ we denote by $(a,b,c)(t,x)$ the functions such that 
$$\left(\begin{array}{c}
T\\e_1\\e_2
\end{array}\right)_t(t,x)=
\left(\begin{array}{ccc}
0 & a & b \\ -a & 0 & c \\ -b &  -c & 0 
\end{array}\right)
\left(\begin{array}{c}
T\\e_1\\e_2
\end{array}\right)(t,x).$$
We first notice that $(a,b,c)(t,0)=(-\beta_x,\alpha_x,-\frac{|\psi|^2}{2}+\frac{a^2}{2t})(t,0)$.
 By computing 
$$T_{tx}=a_xe_1+b_xe_2-(a\alpha+b\beta)T\,\,\,,\,\,\,T_{xt}=\alpha_te_1+\beta_te_2+\alpha e_{1t}+\beta e_{2t},$$  
$$e_{1tx}=-a_xT+c_xe_2-a(\alpha e_1+\beta e_2)-c\beta T\,\,\,,\,\,\,e_{1xt}=-\alpha_t T-\alpha(ae_1+be_2),$$ 
we obtain that
$$\left(\begin{array}{c}
a\\b\\c
\end{array}\right)_x=
\left(\begin{array}{ccc}
0 & 0 & -\beta \\ 0 & 0 & \alpha \\ \beta &  -\alpha & 0 
\end{array}\right)
\left(\begin{array}{c}
a\\b\\c
\end{array}\right)+\left(\begin{array}{c}
\alpha_t\\\beta_t\\0
\end{array}\right),$$
which is equivalent to
$$\left\{\begin{array}{c}
i(\alpha+i\beta)_t +(b-ia)_x-c(\alpha+i\beta)=0,\\
c_x=-\left(\frac{\alpha^2+\beta^2}{2}\right)_x,
\end{array}\right.$$
so we obtain $(a,b,c)(t,x)=(-\beta_x,\alpha_x,-\frac{|\psi|^2}{2}+\frac{a^2}{2t})(t,x)$.
Now we can see that $T$ is a solution of 
$$T_t=-\beta_x e_1+\alpha_x e_2=T\land T_{xx}.$$ 
Therefore, by choosing a point $\chi(t_0,x_0)\in\mathbb R^3$ and  by defining $\chi(t,x)$ as
$$\chi(t,x)=\chi(t_0,x_0)+\int_0^t (T\land T_{xx})(t',x_0)dt'+\int_{x_0}^x T(t,s)ds,$$
we deduce that $\chi$ solves the binormal flow \eqref{binormal}.

In conclusion, given a solution of the cubic 1-D NLS \eqref{NLSa}, we can construct an orthonormal frame $(T,e_1,e_2)$ which leads to a solution of the binormal flow \eqref{binormal}.
Finally we compute the derivatives of the tangent vector and of the normal complex vector $N=e_1+ie_2$ in terms of $\psi$. This will be useful in the following  sections: 
\begin{equation}
\label{2.1}
T_x=\alpha e_1+\beta e_2=\Re\overline \psi N,
\end{equation}

\begin{equation}
\label{2.2}N_x=e_{1x}+ie_{2x}=-\alpha T-i\beta T=-\psi T,
\end{equation}

\begin{equation}
\label{2.3}T_t=\alpha_xe_2-\beta_xe_1=\Im\overline{\psi_x}N,
\end{equation}

\begin{equation}
\label{2.4}N_t=\beta_xT+\left(-\frac{|\psi|^2}{2}+\frac{a^2}{2t}\right) e_2-i\alpha_xT-i\left(-\frac{|\psi|^2}{2}+\frac{a^2}{2t}\right) e_1=-i\psi_x T-i\,\frac{a^2-t|\psi|^2}{2t}\, N.
\end{equation}

\begin{remark} In the case of the Frenet frame  one defines $c$ and $\tau$ from $\psi$ by
$$c(t,x)=|\psi(t,x)|\,\,\,,\,\,\,\tau(t,x)=\Im\frac{\psi_x(t,x)}{\psi(t,x)},$$
then the frame $(T,n,b)$ by
$$
\left(\begin{array}{c}
T\\n\\b
\end{array}\right)_x=
\left(\begin{array}{ccc}
0 & c & 0 \\ -c & 0 & \tau \\ 0 &  -\tau & 0 
\end{array}\right)
\left(\begin{array}{c}
T\\n\\b
\end{array}\right)\,\,\,,\,\,\,\left(\begin{array}{c}
T\\n\\b
\end{array}\right)_t=
\left(\begin{array}{ccc}
0 & -c\tau &  c_x \\ c\tau & 0 & \frac{c_{xx}-c\tau^2}{c} \\ -c_x &  -\frac{c_{xx}-c\tau^2}{c} & 0 
\end{array}\right)
\left(\begin{array}{c}
T\\n\\b
\end{array}\right).$$

One can see the link between these two constructions by considering 
(see \cite {Ha} and also page 5 of \cite{GS}) 

$$e_1(t,x)=\cos\int_0^x \tau(t,s)ds\,n(t,x)-\sin\int_0^x \tau(t,s)ds\,b(t,x)\quad,\quad e_1(0,x)=n(0,x),$$$$e_2(t,x)=\sin\int_0^x \tau(t,s)ds\,n(t,x)+\cos\int_0^x \tau(t,s)ds\,b(t,x)\quad,\quad e_2(0,x)=b(0,x),$$
so
$$
\left(\begin{array}{c}
T\\e_1\\e_2
\end{array}\right)_x=
\left(\begin{array}{ccc}
0 & \alpha & \beta \\ -\alpha & 0 & 0 \\ -\beta &  0 & 0 
\end{array}\right)
\left(\begin{array}{c}
T\\e_1\\e_2
\end{array}\right),$$
with
$$\alpha(t,x)=c(t,x)\cos\int_0^x \tau(t,s)ds\,\,\,,\,\,\beta(t,x)=c(t,x)\sin\int_0^x \tau(t,s)ds.$$
Moreover, one gets that the complex normal vector N is written as
$$N=e_1+ie_2=n(\cos+i\sin)+b(-\sin+i\cos)=(n+ib)(\cos+i\sin)=(n+ib)e^{i\int_0^x \tau(t,s)\,ds}.$$
\end{remark}\bigskip

\section{Asymptotics in space for the tangent vector and the normal vectors}
In this section we shall prove the first part i) of Theorem \ref{errortang}. First we shall prove that the tangent vector, at fixed time, has a limit in space at infinity. Eventually we shall prove that this limit is independent of time. Then we shall do the same for the normal vector $N$ modulated appropriately.  
\subsection{The limit in space for $T(t,x)$ for a fixed given $t$}
\begin{lemma}\label{lemmaTangestx}
Let $0<t\leq 1$. 
There exists a limit $T^\infty(t)$ for $T(t,x)$ as $x$ goes to infinity and
\begin{equation}\label{Tangestx}|T(t,x)-T^\infty(t)|\leq C(a+\|u(1/t)\|_{H^1})\frac{\sqrt t}{x}+C\|\partial_xu(1/t)\|_{L^2}\frac 1{\sqrt{x}}.\end{equation}
\end{lemma}

\begin{proof}
Recall that  \eqref{2.1} gives us $T_x=\Re\overline \psi N$ and that 
$$\psi(t,x)=\frac{e^{i\frac{x^2}{4t}}}{\sqrt t}(a+\overline u)\left(\frac 1t,\frac xt\right).$$ 
In what follows we are going to make a repeated use of integration by parts trying to exploit the high oscillations of the function  $\frac{e^{i\frac{x^2}{4t}}}{\sqrt t}$. From \eqref{2.1} we get
$$\int_x^\infty T_s(t,s)ds=\Re\int_x^\infty\,\overline \psi N (t,s)\,ds=\Re\int_x^\infty\,\frac{e^{-i\frac{s^2}{4t}}}{\sqrt t}(a+ u)\left(\frac 1t,\frac st\right) N (t,s)\,ds$$
$$=-\Re\frac{2t}{-ix}\,\overline \psi(t,x)N (t,x)-\Re\int_x^\infty \frac{2t}{is^2}\,\overline\psi(t,s)N (t,s)\,ds$$
$$-\Re\int_x^\infty e^{-i\frac{s^2}{4t}}\frac{2}{-is\sqrt t}\,(u_s)\left(\frac 1t,\frac st\right) N (t,s)\,ds-\Re\int_x^\infty \frac{2 t}{-is}\,\overline\psi(t,s) N_s (t,s)\,ds.$$
First we notice that from \eqref{2.2} we have $N_x=-\psi T$ so the last term vanishes. Then, since $H^1(\mathbb R)\subset L^\infty(\mathbb R)$ and $|N|=2$ it follows that
\begin{equation}\label{Tangest}
\left|\int_x^\infty T_s(t,s)ds-\Im\int_x^\infty e^{-i\frac{s^2}{4t}}\frac{2}{s\sqrt t}\,(u_s)\left(\frac 1t,\frac st\right) N (t,s)\,ds\right|\leq C(a+\|u(1/t)\|_{H^1})\frac{\sqrt t}{x}.
\end{equation}
This implies the Lemma by using Cauchy-Schwarz inequality.
\end{proof}

\subsection{$T^\infty(t)$ is independent of time}\label{subsectionTindep} 
\begin{lemma}\label{lemmaTanginfty}
The function $T^\infty(t)$ is an independent function of time on $]0,1]$,
$$T^\infty(t)=T^\infty(1)=T^\infty.$$
\end{lemma}

\begin{proof}
Let $0<\epsilon$. We consider $1<x$. Since \eqref{2.3} states $T_t=\Im\overline{\psi_x}N$, we have
$$T(t,x)-T(1,x)=\int_t^1 T_{t'}(t',x)\,dt'=\Im\int_t^1\overline{\psi_x}(t',x)N(t',x)\,dt'$$
$$=\Im\int_t^1\frac{e^{-i\frac{x^2}{4t'}}}{2t'\sqrt {t'}}\,(2(u_x)-ix(a+u))\left(\frac 1{t'},\frac x{t'}\right) N (t',x)\,dt'.$$
Again we will exploit the high oscillations of the function  $\frac{e^{i\frac{x^2}{4t}}}{\sqrt t}$, by integrating by parts 
$$T(t,x)-T(1,x)=\left[\Im e^{-i\frac{x^2}{4t'}}\frac{2\sqrt {t'}}{ix^2}\left(2(u_x)-ix(a+u)\right)\left(\frac 1{t'},\frac x{t'}\right) N (t',x)\right]_t^1$$
$$-\Im\int_t^1e^{-i\frac{x^2}{4t'}}\partial_{t'}\left(\frac{2\sqrt {t'}}{ix^2}\left(2(u_x)-ix(a+u)\right)\left(\frac 1{t'},\frac x{t'}\right) N (t',x)\right)\,dt'.$$
Using  the fact that $u$ and its derivative are bounded, 
$$\left|T(t,x)-T(1,x)\right|\leq \frac{C}{x}(a+\|u\|_{L^\infty_{(1,1/t)}H^2})$$
$$+\left|\Im\int_t^1e^{-i\frac{x^2}{4t'}}\frac{2\sqrt {t'}}{ix^2}\,\partial_{t'}\left(2(u_x)-ix\,u\right)\left(\frac 1{t'},\frac x{t'}\right)N(t',x)\,dt'\right|$$
$$+\left|\Im\int_t^1e^{-i\frac{x^2}{4t'}}\frac{2\sqrt {t'}}{ix^2}(2(u_x)-ix(a+u))\left(\frac 1{t'},\frac x{t'}\right)\left(-i\psi_x T-i\,\frac{a^2-t'|\psi|^2}{2t'}\,  N\right)(t',x)\,dt'\right|.$$
In the last integral we have used the expression \eqref{2.4}: $\partial_{t'}N=-i\psi_x T-i\,\frac{a^2-t'|\psi|^2}{2t'}\, N$. 
Since 
$$\left|\frac{a^2-t'|\psi|^2}{2t'}\right|\leq \frac{2a\|u(1/t')\|_{H^1}+\|u(1/t')\|_{H^1}^2}{t'},$$
the contribution of the last term is of order $\frac{1}{x}$.
Concerning the $-i\psi_x T$ part, recall that 
$$\psi_x(t',x)=\frac{e^{i\frac{x^2}{4t'}}}{2t'\sqrt {t'}}(2(\overline u_x)+ix(a+\overline u))\left(\frac 1{t'},\frac x{t'}\right),$$
so we get again a $\frac{1}{x}$ bound except for the term with no inverse power of $x$. But the integrant of this term is real, so the term vanishes. Let us notice that $u_t$, $u_x$, $u_{tx}$ and $u_{xx}$ are in $L^\infty_{tx}$ because we are assuming  that $u_1\in H^4$ and we shall include  in the upper-bound this dependence. 
Therefore the first integral has a bounded of order $\frac{1}{x}$, except for the term where the derivative in time falls on $u$ and we loose the inverse powers of $x$. Summarizing, we have
$$\left|T(t,x)-T(1,x)\right|\leq \frac{C(u,a,t)}{x}+\left|\Im \int_t^1e^{-i\frac{x^2}{4t'}}\frac{2}{t'\sqrt{t'}}\,(u_x)\left(\frac 1{t'},\frac x{t'}\right)N(t',x)\,dt'\right|,$$
with the constant $C(u,a,t)$ depending on $\|u\|_{L^\infty_{(1,1/t)}H^4}$, $a$ and $t$. 
The integral is of the same type as the first term in the initial expression of $T(t,x)-T(1,x)$, and, as we have seen above, by performing again an integration by parts, it has an upper-bounded of order $\frac{1}{x}$.
Therefore
$$\left|T(t,x)-T(1,x)\right|\leq \frac{C(u,a,t)}{x},$$
with the constant $C(u,a,t)$ depending on $\|u\|_{L^\infty_{(1,1/t)}H^4}$, $a$ and $t$. Notice that for initial data $u_1$ with $\partial_ku_1$ small in $ X^\gamma_1$ for $0\leq k\leq 4$ we have obtained in \cite{BV2} that $\|u\|_{L^\infty_{(1,1/t)}H^4}$ is finite. 
By taking $x$ large with respect to $\|u\|_{L^\infty_{(1,1/t)}H^4}^{-1}$, $a^{-1}$, $t^{-1}$ and to $\epsilon^{-1}$ and by using also Lemma \ref{lemmaTangestx} we obtain that 
$$|T(t,x)-T(1,x)|\leq \epsilon\,\,\,,\,\,\,|T(1,x)-T^\infty(1)|\leq \epsilon\,\,\,,\,\,\,|T(t,x)-T^\infty(t)|\leq \epsilon,$$
so
$$|T^\infty(t)-T^\infty(1)|\leq \epsilon$$
for all $\epsilon>0$ and the Lemma follows. 
\end{proof}
Since the initial data $u_1$ and its space derivatives are small in $ X^\gamma_1$ for $0\leq k\leq 4$ we have from \cite{BV2} that $\|u\|_{L^\infty_{(1,1/t)}H^1}$ is finite.  
Therefore, in view of \eqref{Tangestx}, the first part of (i) in Theorem \ref{errortang}  is proved. Moreover, \eqref{Tangest} becomes
\begin{equation}\label{Tangestbis}
\left|T(t,x)-T^\infty+\Im\int_x^\infty \frac{2}{s\sqrt t}\,(u_s)\left(\frac 1t,\frac st\right) e^{-i\frac{s^2}{4t}}N (t,s)ds\right|\leq C(a+\|u(1/t)\|_{H^1})\frac{\sqrt t}{x}.
\end{equation}

\subsection{The limit in space for $N(t,s)$ for a fixed $t$}
We define the following modulation of the normal vector $N$:
$$\tilde N(t,x)=N(t,x)e^{i\Phi}\,\,\,,\,\,\, \Phi(t,x)=-\frac{a^2}{2}\log t+a^2\log |x|.$$
\begin{lemma}\label{lemmaNasy}
Let $0<t\leq 1$. 
There exists a limit $N^\infty(t)$ for $\tilde N(t,x)$ as $x$ goes to infinity and
\begin{equation}\label{Eestx}\left|\tilde N(t,x)-N^\infty(t)\right|\leq C\|\partial_x u(1/t)\|_{L^2}\frac 1{\sqrt{x}}+\frac{C(1+a^2)(a+\|u(1/t)\|_{H^1})\sqrt t}{x}\end{equation}
$$+\frac{C(a+\|u(1/t)\|_{H^1})\|u(1/t)\|_{L^2}\sqrt t}{\sqrt x}+\frac{C\,(a^2+a^4)\,t}{x^2}.$$
\end{lemma}

\begin{proof}
We shall use formula \eqref{2.2} $N_x=-\psi T$ and we shall perform integration by parts from the oscillating phase $e^{i\frac{x^2}{4t}}$ in $\psi$. We get 
$$\int_x^\infty \tilde N_s(t,s)ds=\int_x^\infty\left(-\psi T+\frac{ia^2}{s}N\right)e^{i\Phi}$$
$$=\frac{2 t}{ix}\,\psi T\,e^{i\Phi}-\int_x^\infty \frac{2t}{is^2}\,\psi Te^{i\Phi}\,ds$$
$$+\int_x^\infty e^{i\frac{s^2}{4t}}\frac{2}{is\sqrt t}\,(\overline u_s)\left(\frac 1t,\frac st\right)T(t,s)\,e^{i\Phi}\,ds+\int_x^\infty e^{i\frac{s^2}{4t}}\frac{2\sqrt t}{is}(a+ \overline u)\left(\frac 1t,\frac st\right) T_s (t,s)\,e^{i\Phi}\,ds$$
$$-a^2\int_x^\infty\frac{2t}{s^2}\,\psi T\,e^{i\Phi}\,ds+\int_x^\infty \frac{ia^2}{s}Ne^{i\Phi}.$$
In view of the $\frac{C(a+\|u(1/t)\|_{H^1})}{\sqrt t}$ bound on $\psi$, the first two terms and the fifth one are upper-bounded by $\frac{C(a+\|u(1/t)\|_{H^1})\sqrt t}{x}$.  Formula \eqref{2.1} insures us that $T_s=\Re\overline{\psi}N$, and by using Cauchy-Schwarz inequality we can upper-bound by $ \frac{C(a+\|u(1/t)\|_{H^1})\|u(1/t)\|_{L^2}\sqrt t}{\sqrt x}$ the part involving $u$ in the fourth term. We get
$$\left|\int_x^\infty \tilde N_s(t,s)\,ds-\int_x^\infty e^{i\frac{s^2}{4t}}\frac{2}{is\sqrt t}\,(\overline u_s)\left(\frac 1t,\frac st\right)T(t,s)\,e^{i\Phi}\,ds\right|$$
$$\leq \frac{C(a+\|u(1/t)\|_{H^1})\sqrt t}{x}+\frac{C(a+\|u(1/t)\|_{H^1})\|u(1/t)\|_{L^2}\sqrt t}{\sqrt x}$$
$$+\left|\int_x^\infty e^{i\frac{s^2}{4t}}\frac{2\sqrt t}{is}\,a\,\Re\left(\frac{ae^{-i\frac{s^2}{4t}}}{\sqrt t}\,N\right)\,e^{i\Phi}\,ds+\int_x^\infty \frac{ia^2}{s}N\,e^{i\Phi}\right|.$$
We obtain then the cancellation of the non-oscillatory terms involving $N$,
$$\left|\int_x^\infty \tilde N_s(t,s)ds-\int_x^\infty e^{i\frac{s^2}{4t}}\frac{2}{is\sqrt t}\,(\overline u_s)\left(\frac 1t,\frac st\right)T(t,s)\,e^{i\Phi}\,ds\right|$$
$$\leq \frac{C(a+\|u(1/t)\|_{H^1})\sqrt t}{x}+\frac{C(a+\|u(1/t)\|_{H^1})\|u(1/t)\|_{L^2}\sqrt t}{\sqrt x}+\left|\int_x^\infty e^{i\frac{s^2}{4t}}\frac{a^2}{is}\,\overline N\,e^{i\Phi}\,ds\right|.$$
By performing a last integration by parts we have
$$\int_x^\infty \frac{e^{i\frac{s^2}{4t}}}{s}\,\overline Ne^{i\Phi}\,ds=-e^{i\frac{x^2}{4t}}\frac{2t}{ix^2}\,\overline Ne^{i\Phi}-\int_x^\infty e^{i\frac{s^2}{4t}}\,\frac{2t}{i}\left(-\frac{2}{s^3}\,\overline N+\frac{ia^2}{s^3}\,\overline N+\frac{1}{s^2}\,\overline N_s\right)e^{i\Phi}.$$
From \eqref{2.2} we have $N_s=-\psi T$, so we get an upper bound of $|N_s|$  of the type $\frac{C(a+\|u(1/t)\|_{H^1})}{\sqrt t}$.
Hence we finally obtain
\begin{equation}\label{Eest0}
\left|\int_x^\infty \tilde N_s(t,s)\,ds-\int_x^\infty e^{i\frac{s^2}{4t}}\frac{2}{is\sqrt t}\,(\overline u_s)\left(\frac 1t,\frac st\right)T(t,s)\,e^{i\Phi}\,ds\right|
\end{equation}
$$\leq\frac{C(1+a^2)(a+\|u(1/t)\|_{H^1})\sqrt t}{x}+\frac{C(a+\|u(1/t)\|_{H^1})\|u(1/t)\|_{L^2}\sqrt t}{\sqrt x}+\frac{C\,(a^2+a^4)\,t}{x^2}.$$
By Cauchy-Schwarz inequality we deduce that  $\tilde N(t,x)$ has a limit $N^\infty(t)$ as $x$ goes to infinity and
the Lemma follows.
\end{proof}

\subsection{$N^\infty(t)$ is independent of time} 
\begin{lemma}\label{lemmaNinfty}
The function $N^\infty(t)$ is an independent function of time on $]0,1]$,
$$T=N^\infty(t)=N^\infty(1)=N^\infty.$$
\end{lemma}
\begin{proof}
Let $0<t\leq 1$, $1<x$. 
We shall use \eqref{2.4}: $\partial_{t'}N=-i\psi_x T-i\,\frac{a^2-t'|\psi|^2}{2t'}\, N$. We obtain
$$\int_t^1 \tilde N_{t'}(t',x)\,dt'=\int_t^1  \left(N_{t'} -\frac{ia^2}{2t'}\,N \right)e^{i\Phi}\,dt'=\int_t^1 \left(-i\psi_x T-i\,\frac{a^2-t'|\psi|^2}{2t'}\,  N-\frac{ia^2}{2t'} N\right)e^{i\Phi}\,dt'$$
$$=\int_t^1 \left((-i)\frac{e^{i\frac{x^2}{4t'}}}{2t'\sqrt {t'}}\,(2(\overline{ u}_x)+ix\,(a+\overline{u}))\left(\frac1{t'},\frac x{t'}\right) T-i\,\frac{a^2-t'|\psi|^2}{2t'}\,  N-\frac{ia^2}{2t'} \,N\right)e^{i\Phi}\,dt'.$$
 As in \S \ref{subsectionTindep}, in the term involving $T$ we perform integrations by parts in time relying on the oscillations of ${e^{i\frac{x^2}{4t}}}$ to obtain
$$|\tilde N(t',x)-\tilde N(1,x)|\leq \frac{C(u,a,t)}{x}+\left|\int_t^1 \left(-i\,\frac{a^2-t'|\psi|^2}{2t'}\,  N-\frac{ia^2}{2t'}\, N\right)e^{i\Phi}\,dt'\right.$$
$$+\left.\int_t^1 e^{i\frac{x^2}{4t'}}\left(\frac {2i}{t'\sqrt{t'}}\,(\overline{u}_x)\left(\frac1{t'},\frac x{t'}\right)T-\frac{2\sqrt {t'}}{x^2}\,(2(\overline{ u}_x)+ix\,(a+\overline{u}))\left(\frac1{t'},\frac x{t'}\right) T_{t'}\right)e^{i\Phi}\,dt'\right|,$$
with the constant $C(u,a,t)$ depending on $\|u\|_{L^\infty_{(1,1/t)}H^4}$, $a$ and $t$. 
In the integral involving $T$ we perform again an integration by parts in time, and we use expression \eqref{2.3}: $T_{t'}=\Im\overline{\psi_x}N$. A bound of type $\frac 1x$ follows also for this part. This is the case also for the last term except its part without an inverse power of $x$, corresponding to the differentiation of the phase in $\overline{\psi_x}$. We have then
$$\left|\tilde N(t',x)-\tilde N(1,x)\right|\leq \frac{C(u,a,t)}{x}+\left|\int_t^1 \left(-i\,\frac{a^2-t'|\psi|^2}{2t'}\, N-\frac{ia^2}{2t'} \,N\right)e^{i\Phi}dt'\right.$$
$$-\left.\int_t^1 e^{i\frac{x^2}{4t'}}\,2\sqrt {t'}i\,(a+\overline{u})\left(\frac1{t'},\frac x{t'}\right) \,\Im \left(\frac{e^{-i\frac{x^2}{4t'}}}{2t'\sqrt{t'}} \,(-i(a+u))\left(\frac1{t'},\frac x{t'}\right)N\right)e^{i\Phi}\,dt'\right|.$$
We recall that $-\frac{|\psi|^2}{2}+\frac{a^2}{2t}$ involves only powers of $u\left(\frac1{t'},\frac x{t'}\right)$, so we get
$$\left|\tilde N(t',x)-\tilde N(1,x)\right|\leq C(u,a,t)\left(\frac{1}{x}+\int_t^1\left|u\left(\frac 1{t'},\frac x{t'}\right)\right|\frac{dt'}{t'}\right)$$
$$+\left|-\int_t^1 \frac{ia^2}{2t'} \,N\,e^{i\Phi}\,dt'+\int_t^1 e^{i\frac{x^2}{4t'}}\,\frac{ia^2}{t'} \,\Re \left(e^{-i\frac{x^2}{4t'}}N\right)\,e^{i\Phi}\,dt'\right|.$$
$$\leq C(u,a,t)\left(\frac{1}{x}+\int_t^1\left|u\left(\frac 1{t'},\frac x{t'}\right)\right|\frac{dt'}{t'}\right)+\left|\int_t^1 e^{i\frac{x^2}{2t'}}\frac{ia^2}{2t'}\, \overline N\,e^{i\Phi}\,dt'\right|.$$
We perform a last integration by parts  of the oscillating function ${e^{i\frac{x^2}{4t}}}$ in the $\overline{N}-$term and use formula \eqref{2.4}: $N_t=-i\psi_xT-i\gamma N$. This way we obtain that this term has also the desired decay $\frac{C(u,a,t)}{x}$. In conclusion
$$\left|\tilde N(t',x)-\tilde N(1,x)\right|\leq C(u,a,t)\left(\frac{1}{x}+\int_t^1\left|u\left(\frac 1{t'},\frac x{t'}\right)\right|\frac{dt'}{t'}\right).$$

Like in Lemma B.1 of \cite{BV1} we can show that if $xu(1)$ and $x\partial_xu$ are in $L^2$ then these regularities are preserved, and the $L^2$ norms of $xu(t)$ and $x\partial_x u(t)$ are controlled by some polynomial growth in time. In particular we can estimate
$$\int_t^1\left|u\left(\frac 1{t'},\frac x{t'}\right)\right|\frac{dt'}{t'}\leq \frac1x\int_t^1\frac x{t'}\,\left|u\left(\frac 1{t'},\frac x{t'}\right)\right|dt'\leq \frac Cx\|xu_1\|_{L^2}^\frac 12\|x\partial_xu_1\|_{L^2}^\frac 12.$$
Therefore
$$\left|\tilde N(t',x)-\tilde N(1,x)\right|\leq \frac{C(u,a,t)}{x},$$
with the constant $C(u,a,t)$ depending on $\|u\|_{L^\infty_{(1,1/t)}H^4}$, $\|xu_1\|_{L^2}$, $\|x\partial_xu_1\|_{L^2}$, $a$ and $t$.
As in \S \ref{subsectionTindep} we conclude that $\tilde N^\infty(t)=\tilde N^\infty(1)=N^\infty$. 
\end{proof}
In particular, \eqref{Eestx} writes
\begin{equation}\label{Eestxindep}\left|\tilde N(t,x)-N^\infty\right|\leq C\|\partial_x u(1/t)\|_{L^2}\frac 1{\sqrt{x}}+\frac{C(1+a^2)(a+\|u(1/t)\|_{H^1})\sqrt t}{x}\end{equation}
$$+\frac{C(a+\|u(1/t)\|_{H^1})\|u(1/t)\|_{L^2}\sqrt t}{\sqrt x}+\frac{C\,(a^2+a^4)\,t}{x}.$$
and \eqref{Eest0} becomes
\begin{equation}\label{Eest}\left|\tilde N(t,x)-N^\infty-i\int_x^\infty e^{i\frac{s^2}{4t}}\frac{2}{s\sqrt t}\,(\overline u_s)\left(\frac 1t,\frac st\right)T(t,s)\,e^{i\Phi}\,ds\right|\leq \frac{C(1+a^2)(a+\|u(1/t)\|_{H^1})\sqrt t}{x}\end{equation}
$$+\frac{C(a+\|u(1/t)\|_{H^1})\|u(1/t)\|_{L^2}\sqrt t}{\sqrt x}+\frac{C\,(a^2+a^4)\,t}{x^2}.$$

Finally, recall that 
$$\tilde N(t,x)=N(t,x)e^{i\Phi(t,x)}=N(t,x)e^{-ia^2\log \frac{\sqrt t}{x}}$$ and we obtain from \eqref{Eestxindep} the second part of (i) in Theorem \ref{errortang}.

\section{The limit of $T(t,x)$ as $t$ goes to 0}
In the next subsection we prove two estimates on the function
$$h(t,s)= e^{-i\frac{s^2}{4t}}\frac{2}{s\sqrt t}\,(u_s)\left(\frac 1t,\frac st\right)e^{-i\Phi},$$
appearing in \eqref{Tangestbis} and \eqref{Eest},
that will allow us to prove the parts (ii)-(iii) of Theorem \ref{errortang} and to analyse $T(t,x)$ as $t$ goes to zero. In subsection \S \ref{subsection limit} we shall prove the existence of a limit of $T(t,x)$ and of $\tilde N(t,x)$ as $t$ goes to zero. 
\subsection{Two integral estimates}

\begin{lemma}\label{L1est} There exists $C>0$ such that for all $t$ small with respect to $u_1$ and $x$, we have
$$\int_x^\infty|h(t,s)|ds\leq  C(\|u_1\|_{X^\gamma_1}+\|\partial_xu_1\|_{X^\gamma_1})+C(a)(\|u_1\|_{X^\gamma_1}+\|\partial_xu_1\|_{X^\gamma_1}+\|xu_1\|_{L^2})\,\frac{t^\frac 14}{x}.$$
\end{lemma}
\begin{proof}
On the one hand, by Cauchy-Schwarz inequality, if $x\geq 1$,
$$\int_1^\infty \frac{2}{s\sqrt t}\left|(u_s)\left(\frac 1t,\frac {s} t\right)\right|\,ds\leq C\left\|\partial_xu\left(\frac 1t\right)\right\|_{L^2}.$$
On the other hand, if $x\leq 1$, we shall introduce the $J$ operator (see Appendix \ref{Jappendix})
$$\int_x^1 \frac{2}{s\sqrt t}\left|(u_s)\left(\frac 1t,\frac st\right)\right|\,ds=\int_\frac xt^\frac 1t \frac{2}{s\sqrt t}\left|(u_s)\left(\frac 1t,s\right)\right|\,ds$$
$$\leq \int_\frac xt^\frac 1t \frac{4\sqrt t}{s}\left(\left|(Ju)\left(\frac 1t,s\right)\right|+\left|s u\left(\frac 1t,s\right)\right|\right)\,ds,$$
so by Cauchy-Schwarz inequality,
$$\int_x^1 \frac{2}{s\sqrt t}\left|(u_s)\left(\frac 1t,\frac st\right)\right|\,ds\leq C\frac tx\left\|Ju\left(\frac 1t\right)\right\|_{L^2}+C\sqrt t\sqrt{\frac{1-x}{t}}\left\|u\left(\frac 1t\right)\right\|_{L^2}$$
$$\leq C\frac tx\left\|Ju\left(\frac 1t\right)\right\|_{L^2}+C\|u_1\|_{L^2}.$$
In Proposition \ref{Jnonlin} we prove that
$$\left\|Ju\left(\frac 1t\right)\right\|_{L^2}\leq C(a)(\|u_1\|_{X^\gamma_1}+\|\partial_xu_1\|_{X^\gamma_1}+\|xu_1\|_{L^2})\,\frac{1}{t^\frac 34},$$
so the Lemma follows.
\end{proof}

\begin{remark}Combining \eqref{Tangestbis} with Lemma \ref{L1est} we obtain that 
for all $x>0$ and $t$ small with respect to $u_1$ and $x$, 
$$|T(t,x)-T^\infty| \leq C(\|u_1\|_{X^\gamma_1}+\|\partial_xu_1\|_{X^\gamma_1})+C(a+\|u_1\|_{X^\gamma_1}+\|\partial_xu_1\|_{X^\gamma_1})\frac{\sqrt t}{x}$$
$$+C(a)(\|u_1\|_{X^\gamma_1}+\|\partial_xu_1\|_{X^\gamma_1}+\|xu_1\|_{L^2})\,\frac{t^\frac 14}{x},$$
and the part of (ii) in Theorem \ref{errortang} follows. 

In \cite{BV2} we have obtained $|\chi(t,x)-\chi(0,x)|\leq C\sqrt t$. As a consequence , for $x,\tilde x>0$ we get
$$|\chi(0,x)-\chi(0,\tilde x)-T^\infty(x-\tilde x)|\leq C\sqrt{t}+|\chi(t,x)-\chi(t,\tilde x)-T^\infty(x-\tilde x)|$$
$$\leq C\sqrt{t}+\left|\int_{\tilde x}^xT(t,s)-T^\infty\,ds\right| \leq C\,\sqrt{t}+\sup_{s\in[\tilde x,x]}|T(t,s)-T^\infty||x-\tilde x|,$$
so the part (iii) of Theorem \ref{errortang} also follows.
\end{remark}

\medskip
\begin{lemma}\label{ansatz} For all $g\in L^\infty$ with $g_s\in L^1$, $0<x\leq\tilde x$,
$$\left|\int_x^{\tilde x} h(t,s) g(s)\,ds-i\int_x^{\tilde x}\widehat{f_+}\left(\frac s2\right) g(s) \,\frac{ds }{s^{ia^2}}\right|$$
$$\leq C(\|u_1\|_{X^\gamma_1}+\|\partial_x u_1\|_{X^\gamma_1}+\|\partial_x^2u_1\|_{X^\gamma_1})(\|g\|_{L^\infty(x,\infty)}+\|g_s\|_{L^1(x,\infty)})\left(\frac{\sqrt{t}}{x}+t^{\frac 16^-}\right).$$
\end{lemma}
\begin{proof}
We obtain by the scattering result \eqref{scatt} applied for Sobolev spaces (see for instance Corollary 3.5 in \cite{BV1})
 $$\int_x^{\tilde x}\frac{1}{s\sqrt t}\left((u_s)\left(\frac 1t,\frac {s} t\right)-\sqrt t\,e^{-i\frac{a^2}{2}\log t}e^{i\frac{s^2}{4t}}\int e^{i\frac{y^2t}{4}}e^{-i\frac{sy}{2}}\,\partial_yf_+(y)\,dy\right) e^{-i\frac{s^2}{4t}}e^{-i\Phi}\,g(s)\,ds$$
$$=\int_x^{\tilde x}\frac{1}{s\sqrt t}\,(r_s)\left(\frac 1t,\frac {s} t\right)e^{-i\frac{s^2}{4t}}e^{-i\Phi}\,g(s)\,ds,$$
with
$$\left\| r\left(\frac 1t\right)\right\|_{H^1}\leq C(\|u_1\|_{X^\gamma_1}+\|\partial_x u_1\|_{X^\gamma_1})\,t^{\frac14^-}.$$
We shall first show that this remainder term can be upper-bounded as in the statement of the Lemma. 
For $x\geq 1$ by Cauchy-Schwarz,
$$\left|\int_x^{\tilde x}\frac{1}{s\sqrt t}\,(r_s)\left(\frac 1t,\frac {s} t\right)e^{-i\frac{s^2}{4t}}e^{-i\Phi}\,g(s)\,ds\right|\leq C(\|u_1\|_{X^\gamma_1}+\|\partial_x u_1\|_{X^\gamma_1})\,\|g\|_{L^\infty(x,\infty)}\,t^{\frac14^-}.$$
Then for $x\geq 1$ we need to treate only the case $\tilde x=1$, and we shall do this by integrating by parts
$$\int_x^{1}\frac{1}{s\sqrt t}\,(r_s)\left(\frac 1t,\frac {s} t\right)e^{-i\frac{s^2}{4t}}e^{-i\Phi}\,g(s)\,ds=\left[\frac{\sqrt t}{s}\,r\left(\frac 1t,\frac {s} t\right)e^{-i\frac{s^2}{4t}}e^{-i\Phi}\,g(s)\right]_x^1$$
$$+\int_x^1\frac{\sqrt t}{s^2}\,r\left(\frac 1t,\frac {s} t\right)e^{-i\frac{s^2}{4t}}e^{-i\Phi}\,g(s)\,ds+\int_x^1\frac{\sqrt t}{s}\,r\left(\frac 1t,\frac {s} t\right)\frac{is}{2t}e^{-i\frac{s^2}{4t}}e^{-i\Phi}\,g(s)\,ds$$
$$+\int_x^1\frac{\sqrt t}{s}\,r\left(\frac 1t,\frac {s} t\right)e^{-i\frac{s^2}{4t}}\frac{ia^2}{s}e^{-i\Phi}\,g(s)\,ds-\int_x^1\frac{\sqrt t}{s}\,r\left(\frac 1t,\frac {s} t\right)e^{-i\frac{s^2}{4t}}e^{-i\Phi}\,g_s(s)\,ds.$$
By a simple integration and  Cauchy-Schwarz we obtain
$$\left|\int_x^{1}\frac{1}{s\sqrt t}\,(r_s)\left(\frac 1t,\frac {s} t\right)e^{-i\frac{s^2}{4t}}e^{-i\Phi}\,g(s)\,ds\right|\leq C\frac{\sqrt {t}}{x}\left\| r\left(\frac 1t\right)\right\|_{L^\infty}\|g\|_{L^\infty}$$
$$+C\left\| r\left(\frac 1t\right)\right\|_{L^2}\|g\|_{L^\infty}+C\frac {\sqrt t}x\left\| r\left(\frac 1t\right)\right\|_{L^2}\|g_s\|_{L^1}$$
$$\leq C(\|u_1\|_{X^\gamma_1}+\|\partial_x u_1\|_{X^\gamma_1})(\|g\|_{L^\infty(x,\infty)}+\|g_s\|_{L^1(x,\infty)})\left(\frac{\sqrt{t}}{x}+t^{\frac 14^-}\right).$$
In conclusion
$$\left|\int_x^{\tilde x} h(t,s)g(s)\,ds-\int_x^{\tilde x} \int e^{i\frac{y^2t}{4}}e^{-i\frac{sy}{2}}\,\partial_yf_+(y)\,dy\,\frac{2g(s)}{s^{1+ia^2}}\,ds\right|$$
$$\leq C(\|u_1\|_{X^\gamma_1}+\|\partial_x u_1\|_{X^\gamma_1})(\|g\|_{L^\infty(x,\infty)}+\|g_s\|_{L^1(x,\infty)})\left(\frac{\sqrt{t}}{x}+t^{\frac 14^-}\right).$$
Since
$$\int e^{i\frac{y^2t}{4}}e^{-i\frac{sy}{2}}\,\partial_yf_+(y)\,dy=i\frac s2\,\widehat{f_+}\left(\frac s2\right)+\int \left(e^{i\frac{y^2t}{4}}-1\right)e^{-i\frac{sy}{2}}\,\partial_yf_+(y)\,dy,$$
it follows that in order to obtain the Lemma it is enough to estimate
$$\int_x^{\tilde x} \int \left(e^{i\frac{y^2t}{4}}-1\right)e^{-i\frac{sy}{2}}\,\partial_yf_+(y)\,dy\,\frac{2g(s)}{s^{1+ia^2}}\,ds$$
$$=\int \mathcal F\left(\left(e^{i\frac{y^2t}{4}}-1\right)\,\partial_yf_+(y)\right)\left(\frac s2\right) \frac{2g(s)\,\mathbb{I}_{(x,\tilde x)}(s)}{s^{1+ia^2}}\,ds$$
$$=\int \left(e^{i\frac{y^2t}{4}}-1\right)\partial_yf_+(y) \int_x^{\tilde x}  \frac{e^{-i\frac{sy}{2}}g(s)}{s^{1+ia^2}}\,ds\,dy=I(x,\tilde x).$$
In the last equality we have used Parseval identity. In all the following the dependence on $u_1$ will come from $\|f_+\|_{H^1
2}$ only, so the dependence of the constants on $u_1$ will be only in terms of $\|\partial_x^ku_1\|_{X^\gamma_1}$ for $0\leq k\leq 2$..\\

We shall need some estimates for $y\neq 0$. One has
$$\int_x^{\tilde x}  \frac{e^{-i\frac{sy}{2}}\,g(s)}{s^{1+ia^2}}ds=\left.\frac{2e^{-i\frac{sy}{2}}\,g(s)}{-iy\,s^{1+ia^2}}\right|_x^{\tilde x}-\int_x^{\tilde x} \frac{2e^{-i\frac{sy}{2}}}{-iy}\left(-\frac{(1+ia^2)\,g(s)}{s^{2+ia^2}}+\frac{g_s(s)}{s^{1+ia^2}}\right)ds,$$
so
\begin{equation}\label{firstest}\left|\int_x^{\tilde x}  \frac{e^{-i\frac{sy}{2}}\,g(s)}{s^{1+ia^2}}\,ds\right|\leq C\frac {\|g\|_{L^\infty}}{|xy|}+C\frac{\|g_s\|_{L^1}}{|y|x}\leq C(g)\,\frac {1}{|y|x},\end{equation}
with
$$C(g)=C(\|g\|_{L^\infty}+\|g_s\|_{L^1}).$$
Also for all $\alpha>0$ and $|y|\geq 1$,
\begin{equation}\label{secondest}\left|\int_x^1 \frac{e^{-i\frac{sy}{2}}\,g(s)}{s^{ia^2}}\,ds\right|\leq C(g)\,C(\alpha)\,\frac {1}{|x|^\alpha|y|}.\end{equation}
Indeed, by integrating by parts
$$\left|\int_x^1 \frac{e^{-i\frac{sy}{2}}\,g(s)}{s^{ia^2}}\,ds\right|\leq C\left(\frac{\|g\|_{L^\infty}}{|y|}+a^2\frac{|\log x|\|g\|_{L^\infty}}{|y|}+\frac{\|g_s\|_{L^1}}{|y|}\right)$$
$$\leq C(\alpha)\,(\|g\|_{L^\infty}+\|g_s\|_{L^1})\,\frac {1}{|x|^\alpha|y|}.$$
It is enough to treat $I(x,1)$ and $I(1,\tilde x)$ for all $0<x\leq 1\leq \tilde x$. 
From \eqref{firstest} and by Cauchy-Schwarz inequality we get
$$|I(1,\tilde x)|\leq C(g)\left(\int_{|y|\leq \frac {1}{\sqrt t}} \left|\frac{e^{i\frac{y^2t}{4}}-1}{y}\right||\partial_yf_+(y)|\,dy+\int_{|y|\geq \frac {1}{\sqrt t}}\left|e^{i\frac{y^2t}{4}}-1\right|\left|\frac{\partial_yf_+(y)}{y}\right|\,dy\right)$$
$$\leq  C(g)\left(\int_{|y|\leq \frac {1}{\sqrt t}} t\,|y\,\partial_yf_+(y)|\,dy+\int_{|y|\geq \frac {1}{\sqrt t}}\left|\frac{\partial_yf_+(y)}{y}\right|\,dy\right)\leq C(g)\,\|\partial_y f_+\|_{L^2}\,t^\frac 14.$$
For treating $I(x,1)$ we need to introduce a cutoff function $\eta(t|y|)$ such that $\eta(r)=1$ for $|r|\leq 1$ and $\eta(r)=0$ for $|r|\geq 2$. On one hand by \eqref{firstest} and by Cauchy-Schwarz inequality
$$\left|\int \left(e^{i\frac{y^2t}{4}}-1\right)\partial_yf_+(y)(1-\eta(t|y|)) \int_x^1 \frac{e^{-i\frac{sy}{2}}\,g(s)}{s^{1+ia^2}}\,ds\,dy\right|$$
$$\leq\int_{\frac 1t\leq |y|}\frac{C(g)}{x|y|}\,|\partial_yf_+(y)|\,dy\leq C(g)\,\|\partial_y f_+\|_{L^2}\,\frac{\sqrt{t}}{x}.$$
On the remaining part of $I(x,1)$ we shall perform an integration by parts
$$\int \left(e^{i\frac{y^2t}{4}}-1\right)\partial_yf_+(y)\,\eta(t|y|) \int_x^1 \frac{e^{-i\frac{sy}{2}}\,g(s)}{s^{1+ia^2}}\,ds\,dy$$
$$=-\int \frac{iyt}{2}e^{i\frac{y^2t}{4}}f_+(y)\,\eta(t|y|) \int_x^1 \frac{e^{-i\frac{sy}{2}}\,g(s)}{s^{1+ia^2}}\,ds\,dy$$
$$-\int \left(e^{i\frac{y^2t}{4}}-1\right)f_+(y)\,\eta_y(t|y|) \int_x^1 \frac{e^{-i\frac{sy}{2}}\,g(s)}{s^{1+ia^2}}\,ds\,dy$$
$$+\int \left(e^{i\frac{y^2t}{4}}-1\right)f_+(y)\,\eta(t|y|) \int_x^1 \frac{ie^{-i\frac{sy}{2}}\,g(s)}{2s^{ia^2}}\,ds\,dy=I_1+I_2+I_3.$$
For $I_1$ and $I_2$ we use \eqref{firstest} and Cauchy-Schwarz inequality
$$|I_1+I_2|\leq \int_{|y|\leq \frac 2t}\frac{C(g) \,t}{x}\, |f_+(y)| \,dy+\int_{\frac 1t\leq |y|\leq \frac 2t}\frac{C(g)\, t}{x|y|}\, |f_+(y)| \,dy\leq C(g)\,\|\partial_y f_+\|_{L^2}\,\frac{\sqrt{t}}{x}.$$
For $0<\alpha<\beta<1$ we split integral $I_3$ into two regions, $|y|\leq t^{-\beta}$ and $t^{-\beta}\leq |y|\leq \frac 2t$. On the first region we upper-bound the integral in $s$ simply by $C\|g\|_{L^\infty}$ and on the other region we use \eqref{secondest} with $\alpha>0$
$$|I_3|\leq C(g)\int_{|y|\leq t^{-\beta}} y^2\,t\,|f_+(y)| \,dy+\int_{t^{-\beta}\leq |y|\leq \frac 2t} |f_+(y)|\,\frac{C(\alpha)\,C(g)}{|x|^\alpha|y|}\,dy.$$
By Cauchy-Schwarz inequality
$$|I_3|\leq C(g)\,\|\partial_y f_+\|_{L^2}\,t^{1-\frac52\beta}+C(\alpha)\,C(g)\,\|\partial_y f_+\|_{L^2}\,\frac{t^\frac\beta 2}{|x|^\alpha}$$
$$\leq (1+C(\alpha))\,C(g)\,\|\partial_y f_+\|_{L^2}\,\left(t^{1-\frac52\beta}+t^{\frac{\beta-\alpha}{2}}\left(\frac{\sqrt{t}}{x}\right)^\alpha\right).$$
We take $\beta=\frac 13$ and $0<\alpha<\frac 13$ so
$$|I_3|\leq\,C(g)\,\|\partial_y f_+\|_{L^2}\,\left(t^\frac16
+t^{\frac 1{6}^-}\left(\frac{\sqrt{t}}{x}\right)^\alpha\right),$$
and the proof of the Lemma is complete.

\end{proof}

\subsection{The existence and properties of $T(0,x)$}\label{subsection limit}

Fix $x>0$ and let $0<t\leq 1$. Let us recall the results of  \eqref{Tangestbis}, \eqref{Eest}, Lemma \ref{L1est} and Lemma \ref{ansatz}:
\begin{equation}\label{Tangestbissimple}
\left|T(t,x)-T^\infty+\Im\int_x^\infty h(t,s)\tilde N (t,s)ds\right|\leq C_1\frac{\sqrt t}{x},
\end{equation}
\begin{equation}\label{Eestsimple}\left|\tilde N(t,x)-N^\infty-i\int_x^\infty \overline{h(t,s)}\,T(t,s)\,ds\right|\leq C_2\left(\frac{\sqrt t}{x}+\frac{t}{x^2}+\sqrt t\right), \end{equation}
\begin{equation}\label{estL1simple}\int_x^\infty|h(t,s)|ds\leq C_3 +C_4\,\frac{t^\frac 14}{x}, \end{equation}
\begin{equation}\label{ansatzsimple}\left|\int_x^{\tilde x} (h(t,s)-\tilde h(s)) \,g(s)\,ds\right|\leq C_5(\|g\|_{L^\infty(x,\infty)}+\|g_s\|_{L^1(x,\infty)})\left(\frac{\sqrt{t}}{x}+t^{\frac 16^-}\right),\end{equation}
with
$$C_1=C(a+\|u_1\|_{X^\gamma}+\|\partial_xu_1\|_{X^\gamma_1}),$$
$$\quad C_2=C(1+a^2)(a+\|u_1\|_{X^\gamma}+\|\partial_xu_1\|_{X^\gamma_1})+C(a+\|u_1\|_{X^\gamma}+\|\partial_xu_1\|_{X^\gamma_1})\|u_1\|_{X^\gamma_1}+C\,(a^2+a^4),$$
$$C_3=C(\|u_1\|_{X^\gamma_1}+\|\partial_xu_1\|_{X^\gamma_1}),\quad C_4=C(a)(\|u_1\|_{X^\gamma_1}+\|\partial_xu_1\|_{X^\gamma_1}+\|xu_1\|_{L^2}),$$
$$C_5=C(\|u_1\|_{X^\gamma_1}+\|\partial_x u_1\|_{X^\gamma_1}+\|\partial_x^2u_1\|_{X^\gamma_1}),$$
and
$$\tilde{h}(s)=i\widehat{f_+}\left(\frac s2\right)  \,\frac{1}{s^{ia^2}}.$$

\begin{lemma}\label{selfsimdecay} The following estimate holds, for $x>0$ and $0<t\leq 1$:
\begin{equation}\label{Tangestter}\left|T(t,x)-T^\infty +\Im N^\infty \int_x^\infty h(t,s)ds+\Re\int_x^\infty h(t,s)\int_s^\infty \overline{h(t,s')}T(t,s')ds'ds\right|\leq C_7(t)\end{equation}
with
$$C_7(t)= C_1\frac{\sqrt t}{x}+C\|\partial_s u(1/t)\|_{L^2}\,C_2\left(\frac{\sqrt t}{x}+\frac{t}{x^2}+\sqrt t\right)$$
$$+C\|u(1/t)\|_{H^1}\left( \left(1+\frac t{x^2}\right)C_2\left(\frac{\sqrt t}{x}+\frac{t}{x^2}+\sqrt t\right)+C_6\left(\frac{\sqrt{t}}{x}+\frac t{x^2}+\frac{t\sqrt{t}}{x^3}\right)\right),$$
and
$$C_6=a+a^2+\|u(1/t)\|_{L^\infty}+\|u(1/t)\|_{L^\infty}^2+\|\partial_xu(1/t)\|_{L^2}.$$
A similar upper bound in terms of positive powers of $\frac{\sqrt{t}}{x}$ holds also for $\tilde N$.
\end{lemma}

\begin{proof}
Combining \eqref{Tangestbissimple}, \eqref{Eestsimple} and \eqref{estL1simple} we obtain
$$\left|T(t,x)-T^\infty +\Im N^\infty \int_x^\infty h(t,s)ds+\Re\int_x^\infty h(t,s)\int_s^\infty \overline{h(t,s')}T(t,s')ds'ds\right|$$
$$\leq C_1\frac{\sqrt t}{x}+\left|\int_x^\infty h(t,s)\,d_0(t,s)\,ds\right|,$$
where
$$d_0(t,x)=\tilde N(t,x)-N^\infty-i\int_x^\infty \overline{h(t,s)}\,T(t,s)\,ds$$
$$=-\frac{2 t}{ix}\,\psi T\,e^{i\Phi}+\int_x^\infty \frac{2t}{is^2}\,\psi Te^{i\Phi}-\frac{2t}{is}\,\psi \,e^{i\Phi}\,\Re \overline{\psi} N+a^2\frac{2t}{s^2}\,\psi T\,e^{i\Phi}-\frac{ia^2}{s}Ne^{i\Phi}\,ds.$$
The expression of $d_0(t,x)$ is obtained in the proof of Lemma \ref{lemmaNasy}, and it was proved that (see for example \eqref{Eestsimple})
\begin{equation}\label{r0infty}
|d_0(t,x)|\leq C_2\left(\frac{\sqrt t}{x}+\frac{t}{x^2}+\sqrt t\right).
\end{equation}
In view of the expression on $d_0$ and $\psi$ we also infer that
\begin{equation}\label{r0l2}
\frac tx\,\left\|\partial_sd_0(t)\right\|_{L^2(\min\{x,1\},1)}\leq C_6\left(\frac{\sqrt{t}}{x}+\frac t{x^2}+\frac{t\sqrt{t}}{x^3}\right),
\end{equation}
for
$$C_6=(a+a^2+\|u(1/t)\|_{L^\infty}+\|u(1/t)\|_{L^\infty}^2+\|\partial_xu(1/t)\|_{L^2}).$$
For $x\geq 1$ by applying Cauchy-Schwarz we get
$$\left|\int_x^\infty h(t,s)\,d_0(t,s)\,ds\right|\leq C\|\partial_s u(1/t)\|_{L^2}\,C_2\left(\frac{\sqrt t}{x}+\frac{t}{x^2}+\sqrt t\right).$$
For $x\leq 1$ we split the integral from $x$ to $1$ and from $1$ to $\infty$, and we perform an integration by parts on $[x,1]$,
$$\int_x^\infty h(t,s)\,d_0(t,s)ds=\int_1^{\infty} h(t,s)\,d_0(t,s) ds+2\sqrt{t}e^{-i\frac{1}{4t}-i\Phi(t,1)}\,u\left(\frac 1t,\frac 1t\right)\,d_0(t,1) $$
$$-\frac{2\sqrt{t}}{x}e^{-i\frac{x^2}{4t}-i\Phi(t,x)}\,u\left(\frac 1t,\frac xt\right)\,d_0(t,x)-\int_x^1u\left(\frac 1t,\frac st\right)\left(\frac{2\sqrt{t}}{s}e^{-i\frac{s^2}{4t}-i\Phi(t,s)}\,d_0(t,s)\right)_s.$$
We use Cauchy-Schwarz and the fact that $u$ belongs to $H^1$ to get
$$
\left|\int_x^\infty h(t,s)\,d_0(t,s)ds\right|\leq C\|u(1/t)\|_{H^1}\left( \left(1+\frac t{x^2}\right)\|d_0(t)\|_{L^\infty(x,\infty)}+\frac tx\,\|\partial_sd_0(t)\|_{L^2(x,1)}\right),
$$
and in view of \eqref{r0l2} the Lemma follows.
Note that this way we have obtained for any $f$ and any $x\neq 0$ the estimate
\begin{equation}\label{main}
\left|\int_x^\infty h(t,s)\,f(t,s)ds\right|\leq C\|u(1/t)\|_{H^1}\left( \left(1+\frac t{x^2}\right)\|f(t)\|_{L^\infty(x,\infty)}+\frac tx\,\|\partial_sf(t)\|_{L^2(x,1)}\right).
\end{equation}
A similar upper bound for $\tilde N$ follows the same by noting that
$$c_0(t,x)=T(t,x)-T^\infty+\Im\int_x^\infty h(t,s)\tilde N(t,s)\,ds$$
$$=-\Re \frac{2 t}{-ix}\,\overline{\psi} N-\Re \int_x^\infty \frac{2t}{is^2}\,\overline{\psi} N+\frac{2t}{-is}\,|\psi|^2 T\,ds,$$
also satisfies
$$\frac tx\,\left\|\partial_sc_0(t)\right\|_{L^2(\min\{x,1\},1)}\leq C_6\left(\frac{\sqrt{t}}{x}+\frac t{x^2}+\frac{t\sqrt{t}}{x^3}\right).$$
\end{proof}

\begin{lemma}\label{lemmaTangsum} 
There exists a constant $C>0$ such that for all $n\in\mathbb N^*$ and $x\neq0$ there exists $a_1(x),...,a_{2n}(x)$ and $R_n(t,x)$ for which the following decomposition holds
\begin{equation}\label{Tangsum}
T(t,x)=\sum_{j=1}^{2n}a_j(x) +R_n(t,x)\end{equation}
	$$+(-1)^n\,\Re\int_x^\infty h(t,s_1)\int_{s_1}^\infty \overline{h(t,s_2)}...\Re\int_{s_{2n-2}}^\infty h(t,s_{2n-1})\int_{s_{2n-1}}^\infty  \overline{h(t,s_{2n})} T(t,s_{2n})\,ds_{2n}...ds_1,$$
with
$$|a_j(x)|\leq C^{j-1} (\|u_1\|_{X^\gamma}+\|\partial_xu_1\|_{X^\gamma_1})^{j-1},$$ and, provided that $\partial_s^ku_1$ are small with respect to $1$ in $X^\gamma_1$ for $0\leq k\leq 2$,
$$R_n(t,x)=O(t^{\frac 16^-}),$$
uniformly in $n$.	
\end{lemma}
\begin{proof}
We prove the Lemma by recursion on $n$. 
We first notice that
$$\|\tilde{h}\|_{L^1}\leq \|\widehat{f_+}\|_{L^1}\leq C\|f_+\|_{H^1}\leq C(\|u_1\|_{X^\gamma_1}+\|\partial_xu_1\|_{X^\gamma_1}).$$
Combining \eqref{Tangestter} with \eqref{ansatzsimple} for $g(s)=1$ and \eqref{estL1simple} we obtain the result for $n=1$ with 
$$a_1(x)=T^\infty\,\,\,,\,\,\,a_2(x)=-\Im N^\infty \int_x^\infty\tilde h(s)\,ds,$$
$$|R_1(t,x)|\leq C_7(t)+C_5\left(\frac{\sqrt{t}}{x}+t^{\frac 16^-}\right).$$
We suppose the result true for $n$ and we shall prove it for $n+1$. By replacing in  \eqref{Tangsum} the tangent $T$ in the integral by its ansatz from \eqref{Tangestter},
$$T(t,x)=\sum_{j=1}^{2n}a_j(x)+R_n(t,x)$$
$$+(-1)^n\Re\int_x^\infty h(t,s_1)\int_{s_1}^\infty \overline{h(t,s_2)}...\times$$
$$\times\Re\int_{s_{2n-2}}^\infty h(t,s_{2n-1})\int_{s_{2n-1}}^\infty  \overline{h(t,s_{2n})}\left(T^\infty +\Im N^\infty\int_{s_{2n}}^\infty h(t,s_{2n+1})\,ds_{2n+1}\right)\,ds_{2n}...ds_1$$
$$+(-1)^{n+1}\Re\int_x^\infty h(t,s_1)\int_{s_1}^\infty \overline{h(t,s_2)}...\times$$
$$\times\Re \int_{s_{2n}}^\infty h(t,s_{2n+1})\int_{s_{2n+1}}^\infty\overline{h(t,s_{2n+2})}T(t,s_{2n+2})\,ds_{2n+2}...ds_1+r_{n+1}(t,x),$$
with
$$|r_{n+1}(t,x)|\leq C_7(t)\left(C_3 +C_4\,\frac{t^\frac 14}{x}\right)^{2n}$$
Since  $\tilde h$ is an $L^1$ function, and since \eqref{estL1simple} yields $h\in L^1$, we can apply \eqref{ansatzsimple} in the iterated integrals as many times as needed to  replace everywhere $h$ by $i\tilde{h}$. We gather the difference terms with $R_n(t,x)$ and obtain $R_{n+1}(t,x)$. This way we get the result for $n+1$ with $a_{2n+1}(x)$ given by
$$(-1)^n\Re\int_x^\infty \tilde{h}(s_1)\int_{s_1}^\infty \overline{\tilde{h}(s_2)}...\Re\int_{s_{2n-2}}^\infty \tilde{h}(s_{2n-1})\int_{s_{2n-1}}^\infty  \overline{\tilde{h}(s_{2n})}T^\infty \,ds_{2n}...ds_1,$$
and with $a_{2n+2}(x)$ given by
$$(-1)^{n+1}\Re\int_x^\infty \tilde{h}(s_1)\int_{s_1}^\infty \overline{\tilde{h}(s_2)}...\times$$
$$\times\Re\int_{s_{2n-2}}^\infty \tilde{h}(s_{2n-1})\int_{s_{2n-1}}^\infty  \overline{\tilde{h}(s_{2n})} \Im N^\infty\int_{s_{2n}}^\infty \tilde{h}(s_{2n+1})\,ds_{2n+1}...ds_1,$$
and
\begin{equation}\label{error}
|R_{n+1}(t,x)|\leq C_7(t)\sum_{k=1}^{n}\left(C_3 +C_4\,\frac{t^\frac 14}{x}\right)^{2k}
\end{equation}
$$+2C_5\left(\frac{\sqrt{t}}{x}+t^{\frac 16^-}\right)\sum_{j=1}^{2n+2}\sum_{k=0}^{j-1}\|\widehat{f_+}\|_{L^1}^k\left(C_3 +C_4\,\frac{t^\frac 14}{x}\right)^{j-1-k}.$$
Finally, for $j\geq 1$,
$$|a_j(x)|\leq \|\widehat{f_+}\|_{L^1}^{j-1}\leq (C\|f_+\|_{H^1})^{j-1}\leq C^{j-1}(\|u_1\|_{X^\gamma_1}+\|\partial_xu_1\|_{X^\gamma_1})^{j-1},$$
and for fixed $x$
$$R_{n}(t,x)=O(t^{\frac 16^-}),$$
provided that $C_3$ and $\|\widehat{f_+}\|_{L^1}$ are small with respect to $1$, so the Lemma follows.
\end{proof}

We shall prove now that there is a limit for $T(t,x)$  and for $\tilde N(t,x)$ as $t$ goes to zero. From the following Proposition the part (iv) of Theorem \ref{errortang} follows. 

\begin{prop}The tangent and the normal vectors $T(t,x)$ and $\tilde N(t,x)$ have a limit at time $t=0$ for $x\neq 0$, and
\begin{equation}\label{Tanglimit}
|T(t,x)-T(0,x)|+|\tilde N(t,x)-\tilde N(0,x)|=O(t^{\frac 16^-}).
\end{equation}
The traces at time $t=0$ have a limit as $x$ goes to infinity, and
\begin{equation}\label{Tanglimitx}
\lim_{x\rightarrow\infty} T(0,x)=T^\infty, \quad \lim_{x\rightarrow\infty} \tilde N(0,x)=N^\infty.
\end{equation}
Moreover, for all $0\leq t\leq 1$ and $x\neq 0$ we have the self similar decay
\begin{equation}\label{ssdecay}
|T(t,x)-T(0,x)|+|\tilde N(t,x)-\tilde N(0,x)|\leq C_8(t,x)\left(1+t^{\frac 16^-}\right),
\end{equation}
with 
$C_8(t,x)$ a linear combination of $\left(\frac{\sqrt {t}}{x}\right)^k$, $1\leq k\leq 4$, with coefficients linear combinations of powers of $\|\partial_s^ju(1)\|_{X^\gamma_1}$, $0\leq j\leq 2$ .
\end{prop}

\begin{proof}
We first notice that for $\partial_s^ku_1$ small enough with respect to $1$ in $X^\gamma_1$ for $0\leq k\leq 1$, we have 
$$\sum_{j=1}^{\infty}|a_j(x)|<\infty.$$
From \eqref{estL1simple} it follows that
$$\left|\Re\int_x^\infty h(t,s_1)\int_{s_1}^\infty \overline{h(t,s_2)}...\Re\int_{s_{2n-2}}^\infty h(t,s_{2n-1})\int_{s_{2n-1}}^\infty  \overline{h(t,s_{2n})} T(t,s_{2n})\,ds_{2n}...ds_1\right|$$
$$\leq \left(C_3 +C_4\,\frac{t^\frac 14}{x}\right)^{2n},$$
and we obtained in Lemma \ref{lemmaTangsum} that
$$\sum_{j=2n}^{\infty}|a_j(x)|\leq\sum_{j=2n}^{\infty}C^{j-1}(\|u_1\|_{X^\gamma_1}+\|\partial_xu_1\|_{X^\gamma_1})^{j-1}.$$
Since $\sum |a_j(x)|$ is finite we can choose $n_t\in\mathbb N$ large enough such that 
$$\sum_{j=2n_t}^{\infty}C^{j-1}(\|u_1\|_{X^\gamma_1}+\|\partial_xu_1\|_{X^\gamma_1})^{j-1}\leq t^\frac 16.$$
By Lemma \ref{lemmaTangsum} we conclude 
$$\left| T(t,x)-\sum_{j=1}^{\infty}a_j(x)\right|=O(t^{\frac 16^-}),$$
and in particular $T(t,x)$ has a limit at $t=0$, 
$$T(0,x)=\sum_{j=1}^{\infty}a_j(x),$$
with the decay in \eqref{Tanglimit}.
We notice that in view of the expression of $a_j(x)$ and of the fact that $\|\widehat {f_+}\|_{L^1}\leq C\|u_1\|_{H^1}<1$, we obtain $T(0)\in L^\infty$  and $T_s(0)\in L^2(\mathbb R\setminus\{0\})$. 
Finally, from Lemma \ref{lemmaTangestx} and Lemma \ref{lemmaTanginfty} we conclude that $T(0,x)$ has a limit as $x$ goes to infinity, and
$$\lim_{x\rightarrow\infty} T(0,x)=T^\infty.$$
Now we focus on $\tilde N(t,x)$ as $t$ goes to zero. 
Estimates \eqref{Tangestbissimple} and \eqref{Eestsimple} allows us to write for $\tilde N$ the estimate 
\begin{equation}\label{Eestter}\left|\tilde N(t,x)-N^\infty +iT^\infty \int_x^\infty \overline{h(t,s)}ds+i\int_x^\infty \overline{h(t,s)}\Im\int_s^\infty h(t,s')\tilde N(t,s')ds'ds\right|\end{equation}
$$\leq C_2\left(\frac{\sqrt t}{x}+\frac{t}{x^2}+\sqrt t\right)+C_1\frac{\sqrt t}{x}\left(C_3 +C_4\,\frac{t^\frac 14}{x}\right).$$
Arguing as above for $T$ we obtain a limit for $\tilde N(t,x)$ for $x>0$ as $t$ goes to zero, with
\begin{equation}\label{Elimit}
|\tilde N(t,x)-\tilde N(0,x)|=O(t^{\frac 16^-}).
\end{equation}
Also, \eqref{Eestx} combined with \eqref{Elimit} implies that $\tilde N(0,x)$ has a limit as $x$ goes to infinity, and
$$\lim_{x\rightarrow\infty} \tilde N(0,x)=N^\infty.$$

Finally, we note that \eqref{estL1simple}, \eqref{ansatzsimple} \eqref{Tangestter} and \eqref{Tanglimit} imply that $T(0,x)$ solves the integral equation
\begin{equation}\label{inteq}
T(0,x)-T^\infty +\Im N^\infty \int_x^\infty \tilde h(s)\,ds+\Re\int_x^\infty \tilde h(s)\int_s^\infty \overline{\tilde h(s')}\,T(0,s')ds'ds=0.\end{equation}
Gathering \eqref{Tanglimit} and  \eqref{inteq} we obtain
$$
|T(t,x)-T(0,x)|\leq C_7(t)+\left| \int_x^\infty h(t,s)-\tilde h(s)\,ds\right|+\left|\int_x^\infty  h(t,s)\int_s^\infty \overline{h(t,s')}\,(T(t,s')-T(0,s'))ds'ds\right|$$
$$+\left|\int_x^\infty h(t,s)\int_s^\infty \overline{h(t,s')-\tilde h(s')}\,T(0,s')ds'ds\right|+\left|\int_x^\infty (h(t,s)-\tilde h(s))\int_s^\infty \overline{\tilde h(s')}\,T(0,s')ds'ds\right|.$$
We use \eqref{ansatzsimple} to estimate the first and the last integral, two times \eqref{ansatzsimple} for the second integral, and  \eqref{main} then \eqref{ansatzsimple} to estimate the third integral
$$|T(t,x)-T(0,x)|\leq C_7(t)+C_5(1+\|\tilde h\|_{L^1})\left(\frac{\sqrt{t}}{x}+t^{\frac 16^-}\right)$$
$$+C\|u(1/t)\|_{H^1}\left( \left(1+\frac t{x^2}\right)\left\|\int_s^\infty \overline{h(t,s')}\,(T(t,s')-T(0,s'))ds'\right\|_{L^\infty(x,\infty)}+\frac tx\,\|\overline{h}(t)(T(t)-T(0))\|_{L^2(x,1)}\right)$$
$$+C\|u(1/t)\|_{H^1}\left( \left(1+\frac t{x^2}\right)\left\|\int_s^\infty \overline{h(t,s')-\tilde h(s')}\,T(0,s')ds'\right\|_{L^\infty(x,\infty)}+\frac tx\,\|\overline{(h(t)-\tilde h)}\,T(0)\|_{L^2(x,1)}\right)$$
$$\leq C_7(t)+C_5(1+\|\tilde h\|_{L^1})\left(\frac{\sqrt{t}}{x}+t^{\frac 16^-}\right)$$
$$+C^2\|u(1/t)\|^2_{H^1}\left(1+\frac t{x^2}\right)\left( \left(1+\frac t{x^2}\right)\|T(t)-T(0)\|_{L^\infty(x,\infty)}+\frac tx\,\|\partial_s(T(t)-T(0))\|_{L^2(x,1)}\right)$$
$$+C_8\frac {t}{x^2}+C\|u(1/t)\|_{H^1}\left(1+\frac t{x^2}\right)C_5\left(\|T(0)\|_{L^\infty(x,\infty)}+\|\partial_sT(0)\|_{L^1(x,\infty)}\right)\left(\frac{\sqrt t}{x}+t^{\frac 16^-}\right)$$
$$+C\|u(1/t)\|_{H^1}\frac tx\,\|h(t)-\tilde h\|_{L^2(x,1)}.$$
We recall that $|T|=1$, that $\partial_s T=\Re\overline\psi N$ and we notice that from \eqref{inteq},
$$\partial_s T(0)=\Im N^\infty\tilde h(x)+\Re\tilde h(x)\int_x^\infty \overline{\tilde h(s)}\,T(0,s)ds.$$
so we have obtained a self similar bound $C_8(t,x)$. The analysis for $\tilde N$ is the same as for $T$.

\end{proof}

\section{The selfsimilar structure}
In this last section we show that the self similar structure is preserved at singularity time $t=0$, in the sense of the statement (v) of Theorem \ref{errortang}.

\begin{prop}The functions $T(0,x)$ and $\tilde N(0,x)$ admit limits on the right and on the left of $x=0$, and their values are, modulo a rotation,
$$T(0,0^\pm)=A_a^{\pm}\quad,\quad \tilde N(0,0^\pm)=B_a^{\pm}.$$
In particular we recover at time zero the angle of the self-similar solution.
\end{prop}

\begin{proof}
Let $t_n$ be a sequence of times that tend to zero, such that $\|u(1/t_n)\|_{L^\infty}$ tends to zero. This is possible since $u\in L^4((1,\infty),L^\infty)$. 
We denote
$$T_n(s)=T(t_n,\sqrt{t_n}\,s)\quad,\quad N_n(s)=N(t_n,\sqrt{t_n}\,s).$$
It follows that
$$T_n'(s)=\sqrt{t_n}\,\Re\left(\overline{\psi}(t_n,\sqrt{t_n}\,s)\,N_n(s)\right)=\Re \left(ae^{i\frac{s^2}{4}}N_n(s)\right)+o(t_n)N_n(s),$$
$$N_n'(s)=-\sqrt{t_n}\,\psi(t_n,\sqrt{t_n}\,s)\,T_n(s)=-ae^{i\frac{s^2}{4}}T_n(s)+o(t_n)T_n(s).$$
Let us recall that $T$ and $N$ are bounded by $1$ and by $2$ respectively. 
In follows that $\mathcal A=\{T_n,n\in\mathbb N\}$ is a collection of pointwise bounded and equicontinuous functions. Then Arzela-Ascoli theorem allows us to obtain a subsequence, that for simplicity we shall denote again $T_n$, that converges uniformly on any compact subset of $\mathbb R$. We can do the same for $\mathcal B=\{N_n,n\in\mathbb N\}$ and conclude that 
$$\lim_{n\rightarrow\infty}(T_n(s),N_n(s))=(T_*(s),N_*(s)).$$
The system satisfied by $(T_*(s),N_*(s))$ is then
$$\left\{\begin{array}{c}T_*'(s)=\Re \left(a e^{i\frac{s^2}{4}}N_*(s)\right),\\ N_*'(s)=ae^{i\frac{s^2}{4}}T_*(s),
\end{array}\right.$$
with initial data $(T_*(0),N_*(0))$, 
which means that 
$$\left(T_*(s),\Re \left(e^{-i\frac{s^2}{4}}N_*(s)\right),\Im\left(e^{-i\frac{s^2}{4}}N_*(s)\right)\right)$$ 
is the Frenet frame of the curve with curvature and torsion $(a,\frac s 2),$
exactly the one of the self-similar profile, see \cite{GRV}. Hence on the one hand, modulo a rotation,
$$T_*(s)=A_a^++\mathcal O\left(\frac{1}{s}\right)\quad,\quad N_*(s)=B_a^++\mathcal O\left(\frac{1}{s}\right).$$
On the other hand, using \eqref{ssdecay}
$$T_*(s)=\lim_{n\rightarrow\infty}T_n(s)=\lim_{n\rightarrow\infty}(T(t_n,\sqrt{t_n}\,s)-T(0,\sqrt{t_n}\,s)+T(0,\sqrt{t_n}\,s))= \mathcal O\left(\frac{1}{s}\right)+\lim_{n\rightarrow\infty}T(0,\sqrt{t_n}\,s),$$
so we obtain the existence and the value of $T(0,0^+)$,
$$T(0,0^+)=A_a^+.$$
In the same maner we get modulo the same rotation that $T(0,0^-)=A_a^-$. Similarly we obtain the existence and the values of $\tilde N(0,0^+)$, $\tilde N(0,0^-)$.
\end{proof}

\section{Appendix: the J-evolution}\label{Jappendix}

At the linear level, if $w(t)=S(t,t_0)w(t_0)$ is the solution of\footnote{In \cite{BV2} we have actually denoted by $u(t)=S(t,t_0)w(t_0)$ the solution of 
$$iu_t+u_{xx}+ \frac{a^2}{t^{1+ 2ia^2}}\overline{u}=0,$$
with initial data $w(t_0)$ at time $t_0$, so $u(t)=e^{-ia^2\log t}w(t)$. Therefore getting estimates on $|\hat u(t)|$ and $|\widehat{J(t)u(t)}|$ is equivalent to getting estimates on $|\hat w(t)|$ and $|\widehat{J(t)w(t)}|$ respectively. }
\begin{equation}\label{lin} iw_t+w_{xx}+\frac{a^2}{2t}(w+\overline w)=0,
\end{equation}
with initial data $w(t_0)$ at time $t_0$, 
then $v(t)=J(t)w(t)=(x+2it\partial_x)\,w(t)$ satisfies
$$iv_t+v_{xx}+\frac{a^2}{2t}(v+\overline v)=\frac{a^2}{2t}(\overline{Jw}-J\overline w)=-2ia^2\,\overline w_x,$$
with initial data $v(t_0)=J(t_0)w(t_0)$ at time $t_0$.

\smallskip

We recall that for the free Schr\"odinger equation, the norm $\|J(t)e^{it\partial_x}f\|_{L^2}$ is constant in time, since $J(t)$ comutes with $e^{it\partial_x^2}$. In here, we do not hope such property for \eqref{lin}, but nevertheless we shall get a control in time better than $t$.

\smallskip

First we shall prove a growth control in time of the Fourier modes of solutions of \eqref{lin}, that improve the one in Lemma 2.1 of \cite{BV2}. More precisely, the parameter $a$ will not be involved  anymore in the polynomial control in time of the growth of the Fourier modes. 
\subsection{Improvement of the growth of the Fourier modes  for the linear equation}

\begin{lemma}\label{lemmaimprov} Let $1\leq t_0\leq t$. For all $\delta>0$ there exists a constant $C(a,\delta)$ such that 
$$|\hat w(t,\xi)|\leq C(a,\delta)\,\frac{t^{\delta}}{t_0^{\delta}}\,( |\hat w(t_0,\xi)|+|\hat w(t_0,-\xi)|)\, \forall \xi\in\mathbb R.$$
\end{lemma}
\begin{proof}
We have
\begin{equation}\label{remodeslin}
\partial_t\,\,\widehat{\Re w}(t,\xi)=\xi^2\,\,\widehat{\Im w}(t,\xi),
\end{equation}
\begin{equation}\label{immodeslin}
\partial_t\,\,\widehat{\Im w}(t,\xi)=-\xi^2\,\,\widehat{\Re w}(t,\xi)+ \frac {a^2}{t}\,\widehat{\Re w}(t,\xi),
\end{equation}
so
\begin{equation}\label{immodeslinsecond}
\partial_t^2\,\,\widehat{\Re w}(t,\xi)=\xi^2\left(-\xi^2+ \frac {a^2}{t}\right)\,\widehat{\Re w}(t,\xi).
\end{equation}
We infer
$$\widehat{\Re w}(t,\xi)=\widehat{\Re w}(t_0,\xi)+(t-t_0)\,\xi^2\,\,\widehat{\Im w}(t_0,\xi)+\int_{t_0}^t(t-\tau)\,\xi^2\left(-\xi^2+ \frac {a^2}{\tau}\right)\,\widehat{\Re w}(\tau,\xi)\,d\tau.$$
Let $\delta>0$, and let $0<\epsilon<\min\{1,a^2\}$ to be chosen also small enough with respect to $\delta$. Then  for $\xi^2\leq \frac{\epsilon}{t}$, 
$$t^{-\delta}\,|\widehat{\Re w}(t,\xi)|\leq t^{-\delta}\left(|\widehat{\Re w}(t_0,\xi)|+|\widehat{\Im w}(t_0,\xi)|\right)+t^{-\delta}\,Ca^2\epsilon\int_{t_0}^t\frac{\tau^\delta}{\tau}\,d\tau\,\sup_{t_0\leq \tau\leq t}\tau^{-\delta}|\widehat{\Re w}(\tau,\xi)|$$
$$\leq t_0^{-\delta}\left(|\widehat{\Re w}(t_0,\xi)|+|\widehat{\Im w}(t_0,\xi)|\right)+\frac{Ca^2\epsilon}{\delta}\,\sup_{t_0\leq \tau\leq t}\tau^{-\delta}\,|\widehat{\Re w}(\tau,\xi)|.$$
Then, by choosing $\epsilon$ small with respect to $\delta$, we obtain
$$|\widehat{\Re w}(t,\xi)|\leq C(a,\delta)\,\frac{t^{\delta}}{t_0^{\delta}}\,( |\hat w(t_0,\xi)|+|\hat w(t_0,-\xi)|).$$
Using similar arguments for the imaginary part we get for $\xi^2\leq \frac{\epsilon}{t}$,
$$t^{-\delta}\,|\widehat{\Im w}(t,\xi)|\leq t^{-\delta}\,|\widehat{\Im w}(t_0,\xi)|+t^{-\delta}C\int_{t_0}^t\frac{a^2}{\tau}\,C(a,\delta)\,\frac{\tau^{\delta}}{t_0^{\delta}}\,( |\hat w(t_0,\xi)|+|\hat w(t_0,-\xi)|)\,d\tau$$
$$\leq t_0^{-\delta}\,|\widehat{\Im w}(t_0,\xi)|+ t_0^{-\delta}\,C(a,\delta)\,( |\hat w(t_0,\xi)|+|\hat w(t_0,-\xi)|),$$
so
$$|\hat w(t,\xi)|\leq C(a,\delta)\,\frac{t^{\delta}}{t_0^{\delta}}\,( |\hat w(t_0,\xi)|+|\hat w(t_0,-\xi)|),$$ 
and the Lemma follows for $\xi^2\leq \frac{\epsilon}{t}$. This part improves  Lemma 2.1 in \cite{BV2}, where the control was of $\frac{t^a}{t_0^a}$.

The proof of Lemma 2.2 in \cite{BV2} contains the result that in the remaining regions $\frac{\epsilon}{t}\leq \xi^2\leq \frac{2a^2}{t}$ and $\frac{2a^2}{t}\leq \xi^2$ the evolution of the $\xi-$Fourier modes stays bounded. For instance,  when $\frac{\epsilon}{t}\leq \xi^2\leq \frac{2a^2}{t}$ we did an energy estimate by considering
$$\partial_t\left(|\widehat{\Re w}(t,\xi)|^2+|\widehat{\Im w}(t,\xi)|^2\right)=\frac{4a^2}{t}\,\Re (\widehat{\Re w}(t,\xi)\,\overline{\widehat{\Im w}(t,\xi)})\leq \frac{2a^2}{t}\left(|\widehat{\Re w}(t,\xi)|^2+|\widehat{\Im w}(t,\xi)|^2\right).$$
By integrating from any $\frac{\epsilon}{\xi^2}\leq t_1\leq\frac{2a^2}{\xi^2}$ to any $\frac{\epsilon}{\xi^2}\leq t\leq \frac{2a^2}{\xi^2}$, we obtain 
$$|\hat w(t,\xi)|^2+|\hat w(t,-\xi)|^2\leq C(a) \left(|\hat w(t_1,\xi)|^2+|\hat w(t_1,-\xi)|^2\right),$$
so the Lemma follows for $\xi^2\leq \frac{2a^2}{t}$. For larger times $t\geq \frac{2a^2}{\xi^2}$, we obtained in \cite{BV2} that the evolution of the Fourier modes is bounded by diagonalizing the system of equations of $\widehat{\Re w}$ and $\widehat{\Im w}$. Therefore the Lemma follows for all $\xi$.\end{proof}
Finally, recall that Lemma 2.2 in \cite{BV2} asserts that
$$|\hat w(t,\xi)|\leq \left(C(a)+C(a,\delta)\frac{1}{(\xi^2t_0)^{\delta}}\right)( |\hat w(t_0,\xi)|+|\hat w(t_0,-\xi)|),\,\, \forall \xi\neq 0.$$

\subsection{J-evolution for the linear equation}
Now we turn to the $J(t)u(t)$ evolution. By using the Duhamel formula for  $S(t,t_0)$ given by equation \eqref{lin},
\begin{equation}\label{DuhJ}
v(t)=S(t,t_0)v(t_0)+\int_{t_0}^t S(t,\tau)(-2ia^2\,\overline w_x(\tau))\,d\tau,
\end{equation}
a similar estimate is obtained also on $v$,
\begin{equation}\label{lowinfty}|\hat v(t,\xi)|\leq C(a,\delta)\,\frac{t^{\delta}}{t_0^{\delta}}\,( |\hat v(t_0,\xi)|+|\hat v(t_0,-\xi)|)+\int_{t_0}^tC(a,\delta)\,\frac{t^{\delta}}{\tau^{\delta}}\,|\xi|\,\frac{\tau^{\delta}}{t_0^{\delta}}\,( |\hat w(t_0,\xi)|+|\hat w(t_0,-\xi)|)\,d\tau\end{equation}
$$\leq C(a,\delta)\,\frac{t^{\delta}}{t_0^{\delta}}\,( |\hat v(t_0,\xi)|+|\hat v(t_0,-\xi)|)+C(a,\delta)\,\frac{t^{\delta}}{t_0^{\delta}}\,t\,|\xi|\,( |\hat w(t_0,\xi)|+|\hat w(t_0,-\xi)|),$$
and we finally obtain
\begin{equation}\label{lowl2}\|\hat v(t,\xi)\|_{L^2(\xi^2\leq \frac {1}{t})}\leq C(a,\delta)\,\frac{t^{\delta}}{t_0^{\delta}}\,\| v(t_0)\|_{L^2}+C(a,\delta)\,\frac{t^{\delta+\frac 12}}{t_0^{\delta}}\,\| w(t_0)\|_{L^2}.
\end{equation}

On the other hand, we get the following version of Lemma 2.2 in \cite{BV2}.
\begin{lemma}\label{highlemma} For all $\xi\neq 0$ and $1\leq t_0\leq t$ the following estimate holds

\begin{equation}\label{high}|\widehat{v}(t,\xi)|\leq \left(C(a)+\frac{C(a,\delta)}{(\xi^2 t_0)^\delta}\right)\left(|\widehat{v}(t_0,\xi)|+|\widehat{v}(t_0,-\xi)|\right)\end{equation}
$$+\left(C(a)+C(a,\delta)\frac{1+|\log|\xi||}{(\xi^2 t_0)^\delta}\right)\,\frac{|\widehat{w}(t_0,\xi)|+|\widehat{w}(t_0,-\xi)|}{|\xi|}.$$
\end{lemma}
\begin{proof}
For $\xi^2\lesssim \frac 1t$ the Lemma follows from \eqref{lowinfty}. For $\xi^2\geq \frac {2a^2}{t}$ we shall diagonalize the system

\begin{equation}\label{remodes}
\partial_t\,\,\widehat{\Re v}(t,\xi)=\xi^2\,\widehat{\Im v}(t,\xi)-2ia^2\xi\,\widehat{\Re w}(t,\xi),
\end{equation}
\begin{equation}\label{immodes}
\partial_t\,\,\widehat{\Im v}(t,\xi)=-\xi^2\,\widehat{\Re v}(t,\xi)+ \frac {a^2}{t}\,\widehat{\Re v}(t,\xi)+2ia^2\xi\,\widehat{\Im w}(t,\xi).
\end{equation}
With similar notations as in \cite{BV2}, we denote for $t\geq 2 a^2$
$$A(t,\xi)=\widehat{\Re v}\left(\frac{t}{\xi^2},\xi\right)\,\,,\,\,B(t,\xi)=\widehat{\Im v}\left(\frac{t}{\xi^2},\xi\right),$$ 
$$Y(t,\xi)=\widehat{\Re w}\left(\frac{t}{\xi^2},\xi\right)\,\,,\,\,Z(t,\xi)=\widehat{\Im w}\left(\frac{t}{\xi^2},\xi\right),$$
so we have the system
\begin{equation}\label{system}
\left\{\begin{array}{c}
\partial_tA(t,\xi)=B(t,\xi)-\frac{2ia^2}{\xi}\,Y(t,\xi),\\\,\\\partial_tB(t,\xi)=\left(-1+\frac{a^2}{ t}\right)A(t,\xi)+\frac{2ia^2}{\xi}\,Z(t,\xi).
\end{array}\right.
\end{equation}
We shall diagonalize the system
$$\partial_t\left(\begin{array}{c}A\\ B\end{array}\right)
=\left(\begin{array}{cc}0 &1\\-\left(1-\frac{a^2}{t}\right) & 0\end{array}\right)\left(\begin{array}{c}A\\ B\end{array}\right)+\frac{2ia^2}{\xi}\left(\begin{array}{c}-Y\\ Z\end{array}\right).$$
Let 
$$\alpha(t)=\sqrt{1-\frac{a^2}{t}}\quad,\quad 
P(t)=\left(\begin{array}{cc}1&1\\i\alpha(t)&-i\alpha(t)\end{array}\right).$$
In particular,
$$\frac{1}{\sqrt{2}}\leq \alpha(t)\leq 1\quad,\quad P^{-1}(t)=\left(\begin{array}{cc} \frac 12&-\frac{i}{2\alpha(t)}\\ \frac 12&\frac{i}{2\alpha(t)}\end{array}\right).$$
Then the new functions
$$\left(\begin{array}{c}A_1(t,\xi)\\ B_1(t,\xi)\end{array}\right)=P^{-1}(t)\left(\begin{array}{c}A(t,\xi)\\ B(t,\xi)\end{array}\right)$$
satisfy
$$\partial_t\left(\begin{array}{c}A_1\\ B_1\end{array}\right)=\partial_t(P^{-1})\,P\left(\begin{array}{c}A_1\\ B_1\end{array}\right)+\left(\begin{array}{cc} i\alpha&0\\0&-i\alpha\end{array}\right)\left(\begin{array}{c}A_1\\ B_1\end{array}\right)+P^{-1}\,\frac{2ia^2}{\xi}\,\left(\begin{array}{c}-Y\\ Z\end{array}\right).$$
We introduce
$$\Phi(t)=t-\frac{a^2}{2}\log t-\int_t^\infty \alpha(s)-1+\frac{a^2}{2s}\,ds,$$
that verifies
$$\Phi(t)'=\alpha(t).$$
Finally, the functions
$$\left(\begin{array}{c}A_2(t,\xi)\\ B_2(t,\xi)\end{array}\right)=
\left(\begin{array}{cc} e^{-i\Phi(t)} &0\\0 & e^{i\Phi(t)}\end{array}\right)\left(\begin{array}{c}A_1(t,\xi)\\ B_1(t,\xi)\end{array}\right)$$
are solutions of
$$
\partial_t\left(\begin{array}{c}A_2\\ B_2\end{array}\right)=M(t)\,\left(\begin{array}{c}A_2\\ B_2\end{array}\right)+\left(\begin{array}{cc} e^{-i\Phi(t)} &0\\0 & e^{i\Phi(t)}\end{array}\right)P^{-1}\,\frac{2ia^2}{\xi}\,\left(\begin{array}{c}-Y\\ Z\end{array}\right)
$$
$$=M(t)\,\left(\begin{array}{c}A_2\\ B_2\end{array}\right)+\frac{2ia^2}{\xi}\left(\begin{array}{c}e^{-i\Phi(t)} (Y-\frac{i}{2\alpha}Z)\\ e^{i\Phi(t)} (Y+\frac{i}{2\alpha}Z)\end{array}\right)$$
where
$$M(t)=\left(\begin{array}{cc} e^{-i\Phi(t)} &0\\0 & e^{i\Phi(t)}\end{array}\right)\partial_t(P^{-1})P\left(\begin{array}{cc} e^{i\Phi(t)} &0\\0 & e^{-i\Phi(t)}\end{array}\right)=\frac{a^2}{4t^2\alpha^2}
\left(\begin{array}{cc} -1 & e^{-2i\Phi(t)}\\e^{2i\Phi(t)} & -1\end{array}\right).$$
By the relation (31) in \cite{BV2}, for  $t\geq 12 a^2$,
\begin{equation}\label{systJ}
\partial_t\left(\begin{array}{c}A_2\\ B_2\end{array}\right)(t,\xi)=M(t)\,\left(\begin{array}{c}A_2\\ B_2\end{array}\right)(t,\xi)+\frac{2ia^2}{\xi}\left(\left(\begin{array}{c}-e^{-2i\Phi(t)}\,  Z^+(\xi)\\ -e^{2i\Phi(t)} \,Y^+(\xi)\end{array}\right)+R(t,\xi)\right),
\end{equation}
where
$$\overline{Y^+(-\xi)}=Z^+(\xi)=\frac 12\,e^{-i\frac{a^2}{2}\log\xi^2}\hat{u_+}(\xi)\,,$$
and
$$\,R(t,\xi)=\left(\begin{array}{c}-e^{-i\Phi(t)}\int_t^\infty\frac{ia^2\,e^{i\Phi(\tau)}}{2\alpha^3(\tau)\,\tau^2}\,Z(\tau,\xi)\,d\tau \\\,\\ e^{i\Phi(t)}\int_t^\infty\frac{ia^2\,e^{-i\Phi(\tau)}}{2\alpha^3(\tau)\,\tau^2}\,Z(\tau,\xi)\,d\tau \end{array}\right).$$
For $2a^2\leq \tilde t\leq t$ we   integrate by parts again. We do it just for  the first component of $R(t,\xi)$ because the other one is similar. We obtain
$$\int_{\tilde t}^t -e^{-i\Phi(\tau)}\int_\tau^\infty\frac{ia^2\,e^{i\Phi(\theta)}}{2\alpha^3(\theta)\,\theta^2}\,Z(\theta,\xi)\,d\theta \,d\tau= \left[e^{-i\Phi(\tau)}\frac{1}{i\alpha(\tau)}\int_\tau^\infty\frac{ia^2\,e^{i\Phi(\theta)}}{2\alpha^3(\theta)\,\theta^2}\,Z(\theta,\xi)\,d\theta\right]_{\tilde t}^t$$
$$-\int_{\tilde t}^t e^{-i\Phi(\tau)}\frac{a^2}{i2\alpha^3(\tau)\tau^2}\,\int_\tau^\infty\frac{ia^2\,e^{i\Phi(\theta)}}{2\alpha^3(\theta)\,\theta^2}\,Z(\theta,\xi)\,d\theta \,d\tau+\int_{\tilde t}^t e^{-i\Phi(\tau)}\frac{1}{i\alpha(\tau)}\frac{ia^2\,e^{i\Phi(\tau)}}{2\alpha^3(\tau)\,\tau^2}\,Z(\tau,\xi) \,d\tau.$$
From Lemma 2.2 in \cite{BV2} it follows that we are in the region where $Z(\tau,\xi)$ is bounded by $C(a)\left(|\hat w(t_0,\xi)|+|\hat w(t_0,-\xi)|\right)$. Moreover, $\frac{1}{\sqrt{2}}\leq \alpha(t)\leq 1$, so
$$\left|\int_{\tilde t}^t -e^{-i\Phi(\tau)}\int_\tau^\infty\frac{ia^2\,e^{i\Phi(\theta)}}{2\alpha^3(\theta)\,\theta^2}\,Z(\theta,\xi)\,d\theta \,d\tau\right|\leq \frac{C(a)}{t}\left(|\hat w(t_0,\xi)|+|\hat w(t_0,-\xi)|\right).$$
Again since $\frac{1}{\sqrt{2}}\leq \alpha(t)\leq 1$, all the entries of $M(t)$ are upper-bounded by $\frac {a^2}{2t^2}$. 
In conclusion, integrating expression \eqref{systJ}, we have for $2a^2\leq \tilde t\leq t$
$$|A_2(t,\xi)|+|B_2(t,\xi)|\leq |A_2(\tilde t,\xi)|+|B_2(\tilde t,\xi)|+\int_{\tilde t}^t\frac {a^2}{t^2}\,(|A_2(\tau,\xi)|+|B_2(\tau,\xi)|)\,d\tau$$
$$+\frac{C(a)}{|\xi|}\left(|\hat{u}_+(\xi)|+|\hat{u}_+(-\xi)|\right)+\frac{C(a)}{t|\xi|} \left(|\hat w(t_0,\xi)|+|\hat w(t_0,-\xi)|\right).$$
So we get 
$$|A_2(t,\xi)|+|B_2(t,\xi)|\leq 2\left(|A_2(\tilde t,\xi)|+|B_2(\tilde t,\xi)|\right)+\frac{C(a)}{|\xi|}|\hat{u}_+(\xi)|+\frac{C(a)}{t|\xi|}\left(|\hat w(t_0,\xi)|+|\hat w(t_0,-\xi)|\right). $$
Finally, from the relation
$$|A_2|^2+|B_2|^2=\left|\frac 12 A-\frac{i}{2\alpha}B\right|^2+\left|\frac 12 A+\frac{i}{2\alpha}B\right|^2= \frac{1}{2}|A|^2+\frac{1}{2\alpha^2}|B|^2,$$
and from $\frac{1}{\sqrt{2}}\leq \alpha(t)\leq 1$ it follows that for $2a^2\leq \tilde t\leq t$,
$$
|A(t,\xi)|^2+|B(t,\xi)|^2\leq C(|A(\tilde t,\xi)|^2+|B(\tilde t,\xi)|^2)+\frac{C(a)}{|\xi|^2}|\hat{u}_+(\xi)|^2+\frac{C(a)}{t^2|\xi|^2} \left(|\hat w(t_0,\xi)|+|\hat w(t_0,-\xi)|\right).
$$
By recovering the first variables and using Lemma 2.10 in \cite{BV2} on the asymptotic state $\hat{u}_+(\xi)$, we obtain the Lemma.
\end{proof}
The pointwise estimate \eqref{high} implies
\begin{equation}\label{highl2}
\|\hat v(t,\xi)\|_{L^2(\frac {1}{t}\leq \xi^2)}\leq C(a,\delta)\frac{t^{\delta}}{t_0^{\delta}}\|v(t_0)\|_{L^2}+C(a,\delta)\,\frac{t^{\delta+\frac12^+}}{t_0^\delta}\|w(t_0)\|_{L^2}.
\end{equation}

In conclusion, gathering \eqref{lowl2} and \eqref{highl2}, we obtain a control for  the $L^2$ norm of the $J-$evolution of the linear solutions,
\begin{equation}\label{l2}
\|J(t)S(t,t_0)f\|_{L^2}\leq C(a,\delta)\frac{t^{\delta}}{t_0^{\delta}}\|J(t_0)f\|_{L^2}+C(a,\delta)\,\frac{t^{\delta+\frac 12^+}}{t_0^\delta}\|f\|_{L^2}.
\end{equation}

\subsection{J-evolution for the nonlinear equation}
We want to show by a bootstrap argument that the solution of the nonlinear equation
$$iu_t+u_{xx}+\frac{a+u}{2t}(|a+u|^2-a^2)=0$$
enjoys a good control in time of $\|J(t)u(t)\|_{L^2}$. First, let us mention that this quantity is finite in time. Indeed, $u(t)\in\dot H^1$ and it was proved in Lemma B.1 in \cite{BV2} that $xu(t)\in L^2$ with a high polynomial growth in time. 
\begin{prop}\label{Jnonlin} If $xu_1\in L^2$ and if $u_1$ is small enough in $X_1^\gamma$, then for all $t\geq 1$ we have
$$\|J(t)u(t)\|_{L^2}\leq C(u_1)\,t^{\frac34},$$
\end{prop}
\begin{proof}
The solution of the nonlinear equation writes as
\begin{equation}\label{nonlinformula}
u(t,x)=S(t,1)\,u_1+\int_{1}^t S(t,\tau) \frac{iF\,(\tau)}{\tau} d\tau.
\end{equation}
with $F(u)$ given by
\begin{equation}\label{F}
F(u)= \frac{|u|^2u+a(u^2+2|u|^2)}{2t}.
\end{equation}
We have from \eqref{l2}
$$t^{-\frac34}\|J(t)S(t,1)u_1\|_{L^2}\leq C(u_1)$$
provided that we choose $\delta<\frac 14.$
Then the worst Duhamel term is the quadratic one. We use again \eqref{l2} with $\delta<\frac 14$,
$$t^{-\frac34}\left\|J(t)\int_1^t S(t,\tau)u^2(\tau)\frac{d\tau}{\tau}\right\|_{L^2}$$
$$\leq C(a,\delta)\,t^{-\frac34}\int_1^t \left(\frac{t^{\delta}}{\tau^{\delta}}\|J(\tau)u^2(\tau)\|_{L^2}+\log t\,\frac{t^{\delta+\frac12^+}}{\tau^\delta}\|u^2(\tau)\|_{L^2}\right)\frac{d\tau}{\tau}.$$
Here $J(t)$ acts on a non-gauge invariant power, so we have to split this term into weight and derivative part, and loose a $t-$power. By using Cauchy-Schwarz inequality
$$t^{-\frac34}\left\|J(t)\int_1^t S(t,\tau)u^2(\tau)\frac{d\tau}{\tau}\right\|_{L^2}\leq C(a,\delta)\, t^{-\frac34}\,t^{\delta}\int_1^t\|xu^2(\tau)\|_{L^2}\,\frac{d\tau}{\tau^{1+\delta}}$$
$$+ C(a,\delta) \,t^{-\frac34}\,t^{\delta}\int_1^t\|u_x(\tau)u(\tau)\|_{L^2}\,\frac{d\tau}{\tau^{\delta}}+C(a,\delta)\,t^{-\frac34}\,\log t\,t^{\delta+\frac 12^+}\|u\|_{L^\infty L^2}\|u\|_{L^\infty L^\infty}$$
\,

$$\leq C(a,\delta) \sup_{1\leq\tau\leq t}\|\tau^{-\frac34}J(\tau)u(\tau)\|_{L^2}\|u\|_{L^\infty H^1}+ C(a,\delta) \|u_x\|_{L^8 L^4}\|u\|_{L^8 L^4}$$
$$+C(a,\delta)t^{-\frac34}\log t\,t^{\delta+\frac 12^+}\|u\|_{L^\infty H^1}^2.$$
In \cite{BV2} it was shown that for small initial data $u_1\in X_1^\gamma$, the solution $u$ satisfies 
$$u\in L^\infty(1,\infty)L^2\cap L^4(1,\infty)L^\infty,$$ 
and implicitly $u$ belongs to all interpolated Strichartz spaces. So provided that $u_1$ and $\partial_x u_1$ are small enough in $X_1^\gamma$
$$t^{-\frac34}\left\|J(t)\int_1^t S(t,\tau)u^2(\tau)\frac{d\tau}{\tau}\right\|_{L^2}\leq \frac 1{3a} \sup_{1\leq\tau\leq t}\|\tau^{-\frac34}J(\tau)u(\tau)\|_{L^2}+C(u_1).$$
The other quadratic term can be treated the same, and we obtain
$$t^{-\frac34}\left\|J(t)\int_1^t S(t,\tau)\frac{a\,u^2(\tau)}{2\tau}\,d\tau\right\|_{L^2}+t^{-\frac34}\left\|J(t)\int_1^t S(t,\tau)\frac{a^2\,|u|^2(\tau)}{\tau}\,d\tau\right\|_{L^2}$$
$$\leq \frac 13 \sup_{1\leq\tau\leq t}\|\tau^{-\frac34}J(\tau)u(\tau)\|_{L^2}+C(u_1).$$
The cubic term is gauge invariant, so by \eqref{l2} with $\delta<\frac 14$ we obtain
$$t^{-\frac34}\left\|J(t)\int_1^t S(t,\tau)|u|^2u(\tau)\frac{d\tau}{\tau}\right\|_{L^2}$$
$$\leq C(a,\delta)\,t^{-\frac34}\int_1^t \left(\frac{t^{\delta}}{\tau^{\delta}}\|J(\tau)u(\tau)\|_{L^2}\|u(\tau)\|_{L^\infty}^2+\log t\,\frac{t^{\delta+\frac12^+}}{\tau^\delta}\|u(\tau)\|_{L^2}\|u(\tau)\|_{L^\infty}^2\right)\frac{d\tau}{\tau}.$$
Again providing that $u_1$ and $\partial_x u_1$ are small enough in $X_1^\gamma$,
$$t^{-\frac34}\left\|J(t)\int_1^t S(t,\tau)|u|^2u(\tau)\frac{d\tau}{\tau}\right\|_{L^2}$$
$$\leq C(a,\delta) \sup_{1\leq\tau\leq t}\|\tau^{-\frac34}J(\tau)u(\tau)\|_{L^2}\|u\|_{L^\infty H^1}^2+ C(a,\delta)\,t^{-\frac34}\,\log t\,t^{\delta+\frac 12^+}\|u\|_{L^\infty H^1}^3$$
$$\leq \frac 16 \sup_{1\leq\tau\leq t}\|\tau^{-\frac34}J(\tau)u(\tau)\|_{L^2}+C(u_1).$$
In conclusion, for all $t\geq 1$ we have

$$\sup_{1\leq\tau\leq t}\|\tau^{-\frac34}J(\tau)u(\tau)\|_{L^2}\leq \frac 23 \sup_{1\leq\tau\leq t}\|\tau^{-\frac34}J(\tau)u(\tau)\|_{L^2}+C(u_1),$$
and the Lemma follows.
\end{proof}


\begin{thebibliography}{99999}

\bibitem{AlKuOk} S. V. Alekseenko, P. A. Kuibin and V. L. Okulov, 
Theory of concentrated vortices. An introduction, 
Springer, Berlin, 2007.

\bibitem{ArHa} R.J. Arms and F.R. Hama, 
{\em Localized-induction concept on a curved vortex and motion of an elliptic vortex
ring}, Phys. Fluids, (1965), 553.

\bibitem{BV0}  V. Banica and L. Vega, 
 {\em On the Dirac delta as initial condition for nonlinear Schr\"odinger equations}, Ann. I. H. Poincar\'e, An. Non. Lin. {\bf 25} (2008), 697-711.

\bibitem{BV1}  V.~Banica and L.~Vega, 
{\it On the stability of a singular vortex dynamics,} 
Comm. Math. Phys. {\bf 286} (2009), 593--627.

\bibitem{BV2}  V.~Banica and L.~Vega, 
{\it Scattering for 1D cubic NLS and singular vortex dynamics,} 
J. Eur. Math. Soc. {\bf 14} (2012), 209--253.

\bibitem{Ba} G.K. Batchelor, 
An Introduction to the Fluid Dynamics, Cambridge University Press, Cambridge,
1967.

\bibitem{Bu}T. F. Buttke, 
 {\em A numerical study of superfluid turbulence in the self-induction approximation}, 
J. Comput. Phys. {\bf 76} (1988), 301--326.




\bibitem{Ca} R. Carles, 
Geometric Optics and Long Range Scattering for One-Dimensional Nonlinear Schr\"odinger Equations,
Comm. Math. Phys. 220 (2001), no. 1, 41-67.


\bibitem{DaR} L. S. Da Rios,  
{\em On the motion of an unbounded fluid with a vortex filament of any shape}, 
Rend. Circ. Mat. Palermo {\bf 22} (1906), 117.





\bibitem{Pa} F. de la Hoz,
 {\em Self-similar solutions for the 1-D Schr\"odinger map on the Hyperbolic plane},
Math. Z. {\bf 257} (2007), 61--80.

\bibitem{dHGV}F. de la Hoz, C.J. Garc\'{\i}a-Cervera and L. Vega.
 {\em A numerical study of the self-similar solutions of the Schr\"odinger map}. SIAM J. Appl. Math. {\bf 70} (2009), 1047--1077. 

\bibitem{DeGe} J. Derezi\'nski and C. G\'erard, 
{\em Scattering theory of classical and quantum $N$-particle systems.}
Texts and Monographs in Physics. Berlin: Springer, 1997.

\bibitem{Dy} M. V. Dyke, An Album of Fluid Motion, Parabolic Press, Stanford, CA, 1982.

\bibitem{FuMi}
Y. Fukumoto and T. Miyazaki,  
{\em Three dimensional distorsions at a vortex filament with axial velocity, }
 J. Fluid Mech. 
 {\bf222}
 (1991)
  396--416. 

\bibitem{GeMaSh} P. Germain, N. Masmoudi and J. Shatah, 
{\em Global solutions for 2D quadratic Schroedinger equations,}  
J. Math. Pures Appl., to appear. 
 
\bibitem{GS} P.G.~Grinevich and M.U.~Schmidt, 
{\it Closed curves in $\mathbb R^3$: a characterization in terms of curvature and torsion, the Hasimoto map and periodic solutions of the Filament Equation,  } arXiv:dg-ga/9703020v1.



\bibitem{GuNaTs}  S. Gustafson, K. Nakanishi and T.-P. Tsai, 
{\em Scattering theory for the Gross-Pitaevskii equation in three dimensions}, 
Comm. Contemp. Maths. {\bf 11} (2009), 657--707.



\bibitem{GRV} S.~Guti\'errez, J.~Rivas and L.~Vega,  
{\it Formation of singularities and self-similar vortex motion under the localized induction approximation, }
Comm. Part. Diff. Eq. {\bf 28} (2003), 927--968.

\bibitem{Ha}
H.~Hasimoto,
 {\it A soliton in a vortex filament, }
J. Fluid Mech.
{\bf 51}
(1972), 477--485.

\bibitem{HaNa} N. Hayashi and P. Naumkin, 
{\em Domain and range of the modified wave operator for Schr\"odinger equations with critical nonlinearity}, 
Comm. Math. Phys. {\bf 267} (2006) 477--492.




\bibitem{JeSm1}
R.L. Jerrard and D. Smets,
{\em On Schr\"odinger maps from $T^1$ to $S^2$,}
arXiv:1105.2736.

\bibitem{JeSm2}
R.L. Jerrard and D. Smets,
{\em On the motion of a curve by its binormal curvature,}
arXiv:1109.5483.

\bibitem{Ko} N. Koiso,
 {\em Vortex filament equation and semilinear Schr\"odinger equation, }
Nonlinear Waves, Hokkaido University Technical Report Series in Mathematics
{\bf 43} (1996)
221--226.

\bibitem{MaBe} A.J. Majda and A.L Bertozzi, 
{\em Vorticity and incompressible flow},
Cambridge Texts in Applied Mathematics, 2002.

\bibitem{MoToTs} K. Moriyama, S. Tonegawa and Y. Tsutsumi, 
{\em Wave operators for the nonlinear Schr\"odinger equation with a nonlinearity of low degree in one or two space dimensions}, 
Comm. Contemp. Math. {\bf 5} (2003), 983--996.

\bibitem{ON}ONERA, A\'erodynamique: d\'ecollement 3D et tourbillons. Tourbillons sur une aile delta en incidence, http://www.onera.fr/conferences/decollement3d/17-tourbillonsailedeltaincidence. html (2004).


\bibitem{Oz} T. Ozawa, 
{\em Long range scattering for nonlinear Schr\"odinger equations in one space dimension}, 
Comm. Math. Phys. {\bf 139} (1991), 479--493. 



\bibitem{Ri}  R.L. Ricca, 
 {\em The contributions of Da Rios and Levi-Civita to asymptotic potential theory and vortex filament dynamics}, 
Fluid Dynam. Res. {\bf 18} (1996), 245--268. 


\bibitem{Sa} P.G. Saffman, 
Vortex dynamics, Cambridge Monographs on Mechanics and Applied Mathematics, 
Cambridge U. Press, New York, 1992.

\bibitem{ShTo} A. Shimomura and S. Tonegawa, 
{\em Long-range scattering for nonlinear Schr\"odinger equations in one and two space dimensions}, J. Differ. Integral Equ. {\bf 17} (2004), 127--150.
\end{thebibliography}
\end{document}